\documentclass[11pt]{article}
\usepackage{amsmath}
\usepackage{amsthm}
\usepackage{amsfonts}
\usepackage{amssymb}
\usepackage{mathtools}
\usepackage[numbers,sort&compress]{natbib}

\usepackage{hyperref}

\usepackage[draft]{todonotes}   

\usepackage{cases}
\usepackage{xcolor}

\newcommand{\dt}{\partial_t}
\newcommand{\Dt}{\mathrm{D}_t}
\newcommand{\dhori}{\partial_h}
\newcommand{\dv}{\mathrm{div}\,}

\newcommand{\dz}[0]{\partial_z}
\newcommand{\nablah}{\nabla_h}
\newcommand{\Deltah}{\Delta_h}
\newcommand{\idx}{\,dx}

\newcommand{\initial}{\text{initial data}}
\newcommand{\moistidx}{\lbrace \mrm{v}, \mrm{c}, \mrm{r} \rbrace}
\newcommand{\thmoistidx}{\lbrace \T, \mrm{v}, \mrm{c}, \mrm{r} \rbrace}
\newcommand{\cb}[0]{\mathcal C_\mrm{b}}
\newcommand{\moistvar}[0]{\lbrace \mathfrak q_\mrm{j}^o \rbrace_{\mrm{j}\in \moistidx}}

\newcommand{\abs}[2]{\bigl| #1 \bigr|^{#2}}
\newcommand{\norm}[2]{\bigl\Arrowvert #1 \bigr\Arrowvert_{#2}}
\newcommand{\Lnorm}[1]{L^{#1}(\Omega)}
\newcommand{\Hnorm}[1]{H^{#1}(\Omega)}

\newcommand{\bHnorm}[1]{H^{#1}(\Gamma)}

\renewcommand{\vec}[1]{\mathbf{#1}}

\newcommand{\mrm}[1]{\mathrm{#1}}

\newtheorem{thm}{Theorem}[section]
\newtheorem{lm}{Lemma}

\theoremstyle{remark}
\newtheorem{remark}{Remark}

\newcommand{\rhod}[0]{\rho_{\mathrm{d}}}
\newcommand{\qv}[0]{q_{\mathrm{v}}}
\newcommand{\qc}[0]{q_{\mathrm{c}}}
\newcommand{\qr}[0]{q_{\mathrm{r}}}
\newcommand{\T}[0]{\mathcal T}

\newcommand{\Sev}[0]{S_{\mathrm{ev}}}
\newcommand{\Scd}[0]{S_{\mathrm{cd}}}
\newcommand{\Sac}[0]{S_{\mathrm{ac}}}
\newcommand{\Scr}[0]{S_{\mathrm{cr}}}

\newcommand{\Vr}[0]{V_{\mathrm{r}}}

\numberwithin{equation}{section}
\allowdisplaybreaks

\title{Local well-posedness of a system coupling compressible atmospheric dynamics and a micro-physics model of moisture in air}

\author{Sabine Doppler\footnote{Fakult\"at f\"ur Mathematik, Universit\"at Wien, Wien, Austria. Email: \href{mailto:sabine.doppler@univie.ac.at}{sabine.doppler@univie.ac.at}},
\,
Rupert Klein\footnote{Institut f\"ur Mathematik, Freie Universit\"at Berlin, Berlin, Germany.
Email: \href{mailto:rupert.klein@math.fu-berlin.de}{rupert.klein@math.fu-berlin.de}}, 
\,
Xin Liu\footnote{
Department of Mathematics, Texas A{\&}M University, College Station, TX 77843, USA. 
Email: \href{mailto:xliu23@tamu.edu}{xliu23@tamu.edu}}, 
\, and \,
Edriss S. Titi\footnote{Department of Mathematics, Texas A{\&}M University, College Station,  TX 77840, USA.  Department of Applied Mathematics and Theoretical Physics, University of Cambridge, Cambridge CB3 0WA, UK.
		Department of Computer Science and Applied Mathematics, Weizmann Institute of Science, Rehovot 76100, Israel. Email: \href{mailto:titi@math.tamu.edu}{titi@math.tamu.edu}\, and \, \href{mailto:Edriss.Titi@maths.cam.ac.uk}{Edriss.Titi@maths.cam.ac.uk}}
}
\date{August 27,2024}
\begin{document}
\allowdisplaybreaks
\maketitle

\begin{abstract}
	We establish here the local-in-time well-posedness of strong solutions with large initial data to a system coupling compressible atmospheric dynamics with a refined moisture model introduced in Hittmeir and Klein \cite{Hittmeir2018} (Theor. Comput. Fluid Dyn. 32:137--164, 2018). This micro-physics moisture model is a refinement of the one communicated by Klein and Majda \cite{Klein2006} (Theor. Comput. Fluid Dyn. 20:525--551, 2006), which serves as a foundation of multi-scale asymptotic expansions for moist deep convection in the atmosphere. Specifically, we show, in this paper, that the system of compressible Navier--Stokes equations coupled with the moisture dynamics is well-posed. In particular, the micro-physics of moisture model incorporates the phase changes of water in moist air. 

\vspace{0.5cm}



\end{abstract}

\tableofcontents

\section{Introduction}

Weather prediction, which affects our daily life closely, has been developed extensively thanks to modern computers. However, due to the complex dynamics involved, it still remains one of the most challenging topics in meteorology, both analytically and numerically. In particular, the thermodynamics of water and dry air, as well as the hydrodynamics coupling, calls for sophisticated modeling and careful mathematical analysis. Our goal in this paper is to perform rigorous analysis on an atmospheric model describing the complex phase changes in the moist atmosphere, and provide a solid mathematical foundation for applications in numerical simulation. 

Water in the atmosphere exists in the basic forms of vapor water, cloud water, and rain water. Briefly speaking, evaporation changes the rain water into vapor water; condensation turns vapor water into cloud water; both accumulation and collection by {rain} convert cloud water into rain water.  Notably, water eventually exits the dynamics in this model by raining. 
In figure \ref{fig:moisture-dyn}, we summarize such dynamics of phase changes in terms of mixing ratios of vapor water, cloud water, and rain water, i.e. $\qv, \qc, \qr $, respectively. 
Such phase changing dynamics are
modeled by the balance equations of vapor water, cloud water, and rain water. Also, in the spirit of Kessler \cite{Kessler1969} 
and Grabowski and Smolarkiewicz \cite{Grabowski1996}, a basic form of the bulk micro-physics closure is taken in Hittmeir and Klein \cite{Hittmeir2018} (see also Klein and Majda \cite{Klein2006}), with which a close system of the flow of moist air is obtained. 

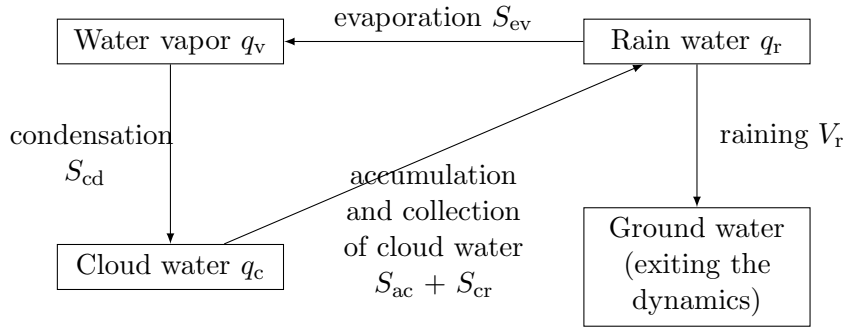
\begin{figure}[h]\label{fig:moisture-dyn}
\centering
\begin {tikzpicture}[-latex, auto, node distance = 3 cm and 7 cm, on grid,
state/.style ={rectangle,
draw, black , text=black, text width = 2.75 cm, align = center}]
\node[state] (A)  {
Water vapor $ \qv $ 
};
\node[state] (B) [below =of A] {{Cloud water $ \qc $}};
\node[state] (C) [right = of A] {{Rain water $ \qr $}};
\node[state] (D) [below =of C] {{Ground water (exiting the dynamics)}};
\path (C) edge [left] node[above] {evaporation $ \Sev $} (A);
\path (A) edge [below]  node[left, text width = 2 cm, align = center] {condensation $ \Scd $} (B);
\path (B) edge [right] node[below,text width = 3 cm, align = center] {accumulation and collection of cloud water $ \Sac + \Scr $} (C);
\path (C) edge [below]  node[right, text width = 2 cm, align = center] {raining $ \Vr $} (D);
\end{tikzpicture}
\caption{Phase changes in moist air}
\end{figure}

Such moisture dynamics is first studied, analytically, in Hittmeir et al. \cite{Hittmeir2017}. Later, coupled with the hydrodynamics provided by the viscous incompressible primitive equations, the same authors in \cite{Hittmeir2019} establish the global well-posedness of strong solutions. See also \cite{Hittmeir2023} for a recent development.

{\it Prior} to the aforementioned studies, considering one moisture quantity, some mathematical analysis has been carried out, for instance, in Coti Zelati et al. \cite{CotiZelati2012,CotiZelati2013,Zelati2015a,Bousquet2014}. Among these works, a Heaviside function is taken to be part of the source term, playing the role of switching between saturated and undersaturated regimes. With differential inclusions and variational techniques, well-posedness and regularity of solutions are investigated.
In Li and Titi \cite{Li2015}, a {coupled} system between the moisture and the barotropic mode and the first baroclinic mode of velocity, proposed by Frierson et al. \cite{Frierson2004}, is shown to be global well-posed with the effect of precipitation. In addition, a relaxation limit is investigated. 
In Majda and Souganidis \cite{Majda2010}, the existence and uniqueness of weak solutions to a linearized hyperbolic system is carried out, together with the involving relaxation limit, which eventually gives rise to a free boundary description of the precipitation fronts.

In this work, we aim at studying the local-in-time well-posedness with large initial data of the original model from \cite{Hittmeir2018} (see also \cite{Klein2006}). Namely, we take into account {the} dynamics of compressible dry air, and consider the compressible Navier-Stokes equations as the hydrodynamics model. See, e.g., \eqref{eq:dry-air-mass-origin}--\eqref{eq:rain-ratio-origin}, below.  
We refer the reader to the introduction as well as the references in 
\cite{Hittmeir2018,Klein2006,Hittmeir2019,Hittmeir2017,Emanuel1994}, for more physical descriptions and considerations of the model studied here. 

{There is a large body of work dedicated to the study of the compressible Navier-Stokes equations}. To mention only a few, the well-known Lions-Feireisl weak solutions can be found in \cite{Lions1998,Feireisl2004}. 
In Cho et al. \cite{Cho2004,Cho2006a}, the local well-posedness of strong solutions, which allows for vacuum, is shown. 
The global well-posedness near equilibrium states is obtained in Matsumura et al. \cite{Matsumura1980,Matsumura1983}. The issue of vacuum in the {compressible} Navier-Stokes equations can be found in Xin et al. \cite{zpxin1998,XinYan2013,Li2017b}. We remark here that, even though vacuum is one of the main issues in the compressible hydrodynamics equations, the examined model in this paper does not allow us to consider vacuum in general (see, e.g., \eqref{eq:rain-ratio}).
{Let $\rho_\mrm d, \rho_\mrm v,  \rho_\mrm c, \rho_\mrm r $ denote the densities of dry air, vapor water, cloud water and rain water, respectively. Then the mixing ratios for vapor water, cloud water, and rain water are defined as $q_\mrm j=\rho_\mrm j/\rho_\mrm d$ for $\mrm j\in \{\mrm v,\mrm c,\mrm r\}$}. The air velocity of the atmosphere is denoted by $ \vec u $. Moreover, $ p, \theta $, and $ \T $ represent the pressure, the potential temperature, and the thermodynamic temperature, respectively. We also take the terminal sedimentation velocity $ \Vr $ to be a given smooth function. 

\bigskip 

Let $ \Dt $ denote the material derivative $ \dt + ( \vec{u} \cdot \nabla ) $. Then without considering the eddy viscosities and diffusions, the governing equations of the flow of moist air in the atmosphere read (cf. \cite{Hittmeir2018})
\begin{align}
	\Dt \rhod + \rhod \dv \vec{u} & = 0
	, \label{eq:dry-air-mass-origin} \\
	\rho \Dt \vec{u} - \rhod \qr \Vr \dz \vec{u} + \nabla p & = 
	- \rho g \vec{e}_3, \label{eq:momentum-origin} \\
	 c_\mathrm{\nu} \T \Dt \log \theta + \sigma \T \Dt \log p  + L(\T) \Dt \qv & = 
	    c_\mathrm{l} \qr \T \Vr (\dz \log \theta {\nonumber} \\
	    & ~~~~ ~~~~ + \dfrac{R_\mathrm{d}}{c_\mathrm{pd}}\dz \log p) , \label{eq:temperature-origin} \\
	  \Dt \qv & = \Sev - \Scd 
	  , \label{eq:vapor-ratio-origin} \\
	  \Dt \qc & = \Scd - ( \Sac + \Scr ) 
	  , \label{eq:cloud-ratio-origin} \\
	  \Dt \qr - \dfrac{1}{\rhod} \dz (\rhod \qr \Vr) & = ( \Sac + \Scr ) - \Sev 
	  , \label{eq:rain-ratio-origin}
\end{align}
where
	\begin{align}
		\text{moist air density} ~ & \rho := \rhod (1 + \qv + \qc + \qr), \label{def:total-density} \\
		\text{pressure} ~ & p := \rhod ( R_\mathrm{d} + R_\mathrm{v} \qv) \T, \label{def:total-pressure} \\
		\text{temperature} ~ &\T  := \theta \biggl( \dfrac{p}{p_{\mathrm{ref}}} \biggr)^{\frac{\gamma-1}{\gamma}}, ~~ \gamma = \dfrac{c_\mathrm{pd}}{c_\mathrm{pd}-R_\mathrm{d}} > 1, \label{def:total-temperature} \\
		\text{velocity} ~ &\vec{u}  := (\vec{v}, w)^\top \in \mathbb R^2 \otimes \mathbb R, \label{def:velocity} \\
		\text{mixed specific heat capacity} ~ &c_\mathrm{\nu}  :=  c_{\mathrm{pd}} + c_{\mathrm{pv}} \qv + c_\mathrm{l} ( \qc + \qr ) , \label{def:mixed-heat-capacity} \\
		\text{mixed gas constant} ~ &\sigma  := \biggl( \dfrac{c_\mathrm{pv}}{c_\mathrm{pd}} R_\mathrm{d} - R_\mathrm{v} \biggr) \qv + \dfrac{c_{\mathrm{l}}}{c_\mathrm{pd}} R_\mathrm{d} ( \qc + \qr) , \label{def:mixed-fluid-constant} \\
		\text{latent heat of condensation} ~ & L(\T) :=  L_\mathrm{ref} + ( c_\mathrm{pv} - c_\mathrm{l} ) ( \T - \T_\mathrm{ref}). \label{def:latent-heat-condensation}
	\end{align}
	The saturation mixing ratio 
	\begin{equation}\label{def:saturation-mixing-ratio}
		q_\mrm{vs}:= q_\mrm{vs}(p, \T)
	\end{equation}
	is a non-negative, uniformly bounded, and Lipschitz continuous function of $ p $ and $ \T $. We also require $ q_\mrm{vs} $ to satisfy
	\begin{equation}\label{cnd:saturation-mixing-ratio}
		0 \leq q_\mrm{vs} \leq q_{\mrm{vs}}^* < \infty , ~~~ \text{and} ~~ q_\mrm{vs} \cdot \T^- = 0. 
	\end{equation}

{As source terms for moisture processes}, we consider the rates of evaporation of rain water, the condensation of vapor water to cloud water and the inverse evaporation process, the auto-conversion of cloud water into rain water by accumulation of microscopic droplets, and the collection of cloud water by falling rain, defined by 
\begin{equation}\label{def:state-variable-moist}
\begin{aligned}
	\Sev & := c_\mathrm{ev} \dfrac{p}{\rho} ( q_\mathrm{vs} - \qv) ^+ \qr  \\
	& ~ = c_\mathrm{ev} \T \dfrac{(R_\mathrm{d}+R_\mathrm{v}\qv)}{(1+\qv + \qc + \qr)} ( q_\mathrm{vs} - \qv) ^+ \qr , \\
	\Scd & := c_\mathrm{cd} ( \qv - q_\mathrm{vs} ) \qc + c_\mathrm{cn} ( \qv - q_\mathrm{vs} )^+ q_\mathrm{cn}, \\
	\Sac & :=  c_\mathrm{ac} ( \qc - q_\mathrm{ac})^+,  \\
	\Scr & :=  c_\mathrm{cr} \qc \qr,
\end{aligned}
\end{equation}
respectively, in the spirit of \cite{Kessler1969,Grabowski1996}. 

In the above, $ R_\mrm d, R_\mrm v, c_\mrm{pd}, c_\mrm{pv}, c_\mrm{l}, c_\mrm{ev}, c_\mrm{cd}, c_\mrm{cn}, c_\mrm{ac}, c_\mrm{cr}, p_\mrm{ref}, L_\mrm{ref}, \T_\mrm{ref}, q_\mrm{vs}^*, q_\mrm{ac}  $ are assumed to be given positive constants, with $c_\mrm{pd} > R_\mrm d$. Moreover, for any function $ f $, the non-negative and non-positive parts are defined in \eqref{def:sign-op}, below. 
We refer the reader to \cite{Hittmeir2018} for the physical terminologies and meanings of these quantities. 

In this paper, we will consider $ \rhod, \vec{u}, \T, \qv, \qc, \qr $ to be our unknowns, and we will write down the equations that govern them in the following. 

Using \eqref{def:total-pressure} and \eqref{def:total-temperature}, one can write down
\begin{align*}
	\theta & = \dfrac{ \T \cdot p_\mathrm{ref}^{\frac{\gamma-1}{\gamma}} }{ p^{\frac{\gamma-1}{\gamma}} } = \dfrac{ \T^{\frac{1}{\gamma}} p_\mathrm{ref}^{\frac{\gamma-1}{\gamma}} }{ \rhod^{\frac{\gamma-1}{\gamma}} (R_\mathrm{d} + R_\mathrm{v} \qv )^{\frac{\gamma-1}{\gamma}} }.
\end{align*}
Therefore, one can derive, from \eqref{eq:dry-air-mass-origin}, \eqref{eq:temperature-origin}, and \eqref{eq:vapor-ratio-origin}, that $ \T $ satisfies, 
\begin{equation}\label{eq:temperature-origin-2}
	\begin{aligned}
	 & \biggl( \dfrac{c_\mathrm{\nu}}{\gamma} + \sigma\biggr) [ \dt \T + (\vec{u}\cdot\nabla) \T ]  = \biggl(\sigma - \dfrac{R_\mathrm{d}}{c_\mathrm{pd}}c_\mathrm{\nu} \biggr) \T \dv \vec{u} \\
	 & ~~ - \biggl\lbrack \biggl(\sigma - \dfrac{R_\mathrm{d}}{c_\mathrm{pd}}c_\mathrm{\nu} \biggr) \dfrac{ R_\mathrm{v} \T}{R_\mathrm{d} + R_{\mathrm{v}}\qv } + L(\T) \biggr\rbrack ( \Sev - \Scd ) \\
	 & ~~ + c_\mathrm{l} \qr \Vr \dz \T,  \\
	\end{aligned}
\end{equation}
where we have used the fact that 
$$
\dz \log \theta + \dfrac{R_\mathrm{d}}{c_\mathrm{pd}}\dz \log p = \dz \log \T.
$$


Consequently,
after putting all equations together and adding back the eddy viscosities and diffusions, we arrive at the following system of equations:

\begin{gather}
	 \dt \rhod + \dv (\rhod \vec{u}) =  0, \label{eq:dry-air-mass} \\
	 \begin{gathered} \rhod Q_\mathrm{m} [\dt \vec{u} + (\vec{u} \cdot \nabla) \vec{u}]  - \rhod \qr \Vr \dz \vec{u} + \nabla p = \mu \Delta \vec{u} + (\mu + \lambda) \nabla \dv \vec{u} \\
	  - \rhod Q_\mathrm{m} g \vec{e}_3, \end{gathered} \label{eq:momentum} \\
	 \begin{gathered}
		 Q_\mathrm{th} [ \dt \T + (\vec{u}\cdot\nabla) \T ]  - c_\mathrm{l} \qr \Vr \dz \T = \kappa \Delta \T 
		 + Q_\mathrm{cp} \T \dv \vec{u} \\
	  	 - ( Q_1 \T + Q_2 ) ( \Sev - \Scd  ) ,\end{gathered} \label{eq:temperature} \\
	  \dt \qv + (\vec{u}\cdot\nabla) \qv  = \Sev - \Scd + \Delta \qv, \label{eq:vapor-ratio} \\
	  \dt \qc + (\vec{u}\cdot\nabla) \qc  = \Scd - ( \Sac + \Scr ) + \Delta \qc, \label{eq:cloud-ratio} \\
	  \dt \qr + (\vec{u}\cdot\nabla) \qr - \dz(\qr \Vr) - \qr \Vr \dz \log \rhod = ( \Sac + \Scr ) - \Sev + \Delta \qr, \label{eq:rain-ratio}
\end{gather}
where $ \mu > 0 , ~ 2\mu + 3\lambda > 0$, and $ \kappa > 0 $ denote the shear viscosity constant, the bulk viscosity constant, and the heat conductivity constant, respectively. Here, we have used the following notations: 
\begin{equation}\label{def:state-variable}
	\begin{gathered}
		Q_\mathrm{m} := 1 + \qv + \qc + \qr, \\
		Q_\mathrm{th} := \dfrac{c_\mathrm{\nu}}{\gamma} + \sigma = \dfrac{c_\mathrm{pd}}{\gamma} + \biggl( \dfrac{c_\mathrm{pv}}{\gamma} + \dfrac{c_\mathrm{pv}}{c_\mathrm{pd}}R_\mathrm{d} - R_\mathrm{v} \biggr) \qv  + \biggl( \dfrac{c_\mathrm{l}}{\gamma} + \dfrac{c_\mathrm{l}}{1c_\mathrm{pd}}R_\mathrm{d} \biggr)(\qc + \qr ), \\
		Q_\mathrm{cp} := \sigma - \dfrac{R_\mathrm{d}}{c_\mathrm{pd}} c_\mathrm{\nu} = - R_\mathrm{d} - R_\mathrm{v} \qv, \\
		Q_1 := \dfrac{R_\mathrm{v}}{R_\mathrm{d} + R_\mathrm{v} q_\mathrm{v}} \biggl( \sigma - \dfrac{R_\mathrm{d}}{c_\mathrm{pd}} c_\mathrm{\nu} \biggr) + c_\mathrm{pv} - c_\mathrm{l} =  c_\mathrm{pv} - c_\mathrm{l}- R_\mathrm{v} , \\
		Q_2 := L_\mathrm{ref} - ( c_\mathrm{pv} - c_\mathrm{l} ) \T_\mathrm{ref}.
	\end{gathered}
\end{equation}


We consider \eqref{eq:dry-air-mass} -- \eqref{eq:rain-ratio} in the domain $ \Omega := 2 \mathbb{T}^2 \times (0,1) \subset \mathbb{R}^3 $, where $ 2 \mathbb{T}^2 \subset \mathbb{R}^2 $ is the periodic domain with two periods in either direction on the two-dimensional plane. On the boundaries $ \lbrace z = 0, 1 \rbrace $, the following boundary conditions are imposed for the above system of equations:
\begin{align}
	\text{stress free and no penetration} \qquad \partial_{z} \vec{v}\big|_{z=0,1} & = 0, ~ w\big|_{z=0,1}  = 0, \label{bc:momentum} \\
	\text{heat flux} \qquad \partial_{z} \T \big|_{z=0,1} & = \alpha_{\T}^b(\T^{b} - \T)\big|_{z=0,1}, \label{bc:temperature} \\
	\text{moisture flux} \qquad \partial_{z} q_\mathrm{j} \big|_{z=0,1} & = \alpha_\mathrm{j}^b(q_{\mathrm{j}}^{b} - q_\mathrm{j})\big|_{z=0,1}, ~ \mathrm{j} \in \lbrace \mathrm{v}, \mathrm{c}, \mathrm{r} \rbrace, \label{bc:moisture}
\end{align}
with some constants $ \alpha^b_{\mathrm{j}}\big|_{z=0,1} $, $ \mathrm{j} \in  \lbrace \T, \mathrm{v}, \mathrm{c}, \mathrm{r} \rbrace $, satisfying
\begin{equation}\label{bc:sign}
	\alpha_\mrm{j}^b \geq 0 ~ \text{on} ~ \lbrace z = 1 \rbrace, ~~\text{and} ~ 
	\alpha_\mrm{j}^b \leq 0 ~ \text{on} ~ \lbrace z = 0 \rbrace. 
\end{equation}

We assume, in this paper, that $ \alpha_{\T}^b, \T^b, \lbrace \alpha_\mrm{j}^b, q_\mrm{j}^b\rbrace_{\mrm{j}\in \moistidx} $ are given bounded and smooth functions. 

\begin{remark}
Condition \eqref{bc:sign} is necessary to show the non-negativity of $ \T $ and of $ \lbrace q_\mrm{j} \rbrace_{\mrm j \in \moistidx} $.
\end{remark}

Now we state the main theorem of this paper:
\begin{thm}\label{thm}
	Let the initial data $ (\rho_{\mrm d, in}, \vec u_{in}, \mathcal T_{in}, \lbrace q_{\mrm j, in} \rbrace_{\mrm j \in \moistidx}) $ be sufficiently smooth, satisfying \eqref{bc:momentum}--\eqref{bc:moisture}, \eqref{def:initial-bound}, and \eqref{def:initial-bound-2}, below. Then, there exist a finite time $ T^* \in (0,\infty) $, depending on the initial data, and a unique strong solution to equations \eqref{eq:dry-air-mass}--\eqref{eq:rain-ratio}, satisfying the following regularity and estimates: 
	\begin{gather*}
		\rhod^{1/2}, \log \rhod \in L^\infty(0,T^*;\Hnorm{2}), \\
		 \dt\rhod^{1/2}, \dt \log \rhod \in L^\infty(0,T^*;\Hnorm{1}),\\
		 \vec u, \T, \lbrace q_\mrm{j} \rbrace_{\mrm{j} \in\moistidx} \in L^\infty(0,T^*;\Hnorm{2})\cap L^2(0,T^*;\Hnorm{3}), \\
		\dt \vec u, \dt  \T, \lbrace \dt q_\mrm{j} \rbrace_{\mrm{j} \in\moistidx}\in L^\infty(0,T^*;\Lnorm{2}) \cap L^2(0,T^*;\Hnorm{1}),
	\end{gather*}
	and
	\begin{gather*}
	 \sup_{0\leq s \leq T^*} \bigl(\norm{\rhod^{1/2}(s), \log \rhod (s), \vec u (s), \T (s), \lbrace q_\mrm j (s)\rbrace_{\mrm j \in \moistidx} }{\Hnorm{2}}^2 \\
	 ~~~~ + \norm{\dt \rhod^{1/2}(s), \dt \log \rhod(s)}{\Hnorm{1}}^2 \\
	 ~~~~ + 
	\norm{\dt \vec u(s), \dt \T(s), \lbrace \dt q_\mrm j (s) \rbrace_{\mrm j \in \moistidx}}{\Lnorm{2}}^2 \bigr) \\
	 + \int_0^{T^*} \norm{\vec u (s), \T (s), \lbrace q_\mrm j (s)\rbrace_{\mrm j \in \moistidx} }{\Hnorm{3}}^2 \\
	 ~~~~ + 
	\norm{\dt \vec u(s), \dt \T(s), \lbrace \dt q_\mrm j (s) \rbrace_{\mrm j \in \moistidx}}{\Hnorm{1}}^2 \,ds  \leq \mathcal{K} < \infty,
	\end{gather*}
	where $ \mathcal{K} $ depends only on initial data. In addition, the solution is continuously depending on the initial data in a manner that is described in \eqref{stability:001} and \eqref{stability:002}, below. 
\end{thm}
\begin{proof}
	The proof of Theorem \ref{thm} is divided into sections \ref{sec:approximation}--\ref{sec:well-posedness}. 
\end{proof}

\bigskip

Before moving to the proof of Theorem \ref{thm}, we comment on the strategy here. Notice that $ \T, \lbrace q_\mrm j \rbrace_{\mrm j \in \mrm v, \mrm c, \mrm r} $ representing the temperature, the mixing ratios of vapor, cloud, and rain, which are non-negative. To capture such physical properties, instead of system \eqref{eq:dry-air-mass}--\eqref{eq:rain-ratio}, we introduce the approximation system \eqref{eq:dry-air-mass-app}--\eqref{eq:rain-ratio-app}, below. Then once the well-posedness of such an approximation system is obtained, after showing the non-negativity of the aforementioned variables, the solution that solves the approximation system also solves system \eqref{eq:dry-air-mass}--\eqref{eq:rain-ratio}. This is done in section \ref{sec:non-negativity}.

Moreover, to handle the inhomogeneous boundary conditions \eqref{bc:temperature} and \eqref{bc:moisture}, we introduce the new variables $ \mathfrak T, \lbrace \mathfrak q_\mrm j\rbrace_{\mrm j \in \moistidx} $ in \eqref{def:homo-variable}, below, which satisfy \eqref{eq:temperature-app-1}--\eqref{eq:rain-ratio-app-1} with homogeneous boundary conditions \eqref{bc:app-hom}, below. Therefore, solving the approximation system \eqref{eq:dry-air-mass-app}--\eqref{eq:rain-ratio-app} is equivalent to solving the system of equations \eqref{eq:dry-air-mass-app}--\eqref{eq:momentum-app} and \eqref{eq:temperature-app-1}--\eqref{eq:rain-ratio-app-1}, below. 

Consequently, the proof of Theorem \ref{thm} is reduced to establishing the well-posedness of \eqref{eq:dry-air-mass-app}--\eqref{eq:momentum-app} and \eqref{eq:temperature-app-1}--\eqref{eq:rain-ratio-app-1} with the homogeneous boundary conditions \eqref{bc:momentum} and \eqref{bc:app-hom}. This is done through a fixed point argument. To be more precise, the solution to the approximation system is constructed as the fixed point of a contraction mapping in the adequate topology. To define the contraction mapping, we introduce the associated linear system of equations \eqref{eq:dry-air-mass-lin}--\eqref{eq:rain-ratio-lin} with the boundary conditions \eqref{bc:app-lin}, below, by freezing some terms in \eqref{eq:dry-air-mass-app}--\eqref{eq:momentum-app} and \eqref{eq:temperature-app-1}--\eqref{eq:rain-ratio-app-1}. The corresponding contraction mapping $ \mathcal M $ is then defined as the mapping that maps the frozen variables to the solutions of the linear system (see \eqref{def:mapping}, below). By showing the contracting property of $ \mathcal M $ in section \ref{sec:contracting}, we obtain the unique solution to the nonlinear approximation system of equations \eqref{eq:dry-air-mass-app}--\eqref{eq:momentum-app} and \eqref{eq:temperature-app-1}--\eqref{eq:rain-ratio-app-1}. Then the proof is finished by establishing the well-posedness of such solutions. 

\smallskip 

The rest of this paper is organized as follows. In section \ref{sec:notation}, we introduce some notations being used hereafter. In section \ref{sec:approximation}, we introduce an approximation system to \eqref{eq:dry-air-mass}--\eqref{eq:rain-ratio}. Section \ref{sec:linear_contract} is dedicated to study the associated linear system, where the adequate bound of solutions and the {contraction} mapping property are shown. In section \ref{sec:well-posedness}, we finish the proof of our well-posedness theorem. In section \ref{appendix}, we collect the required estimates of the nonlinearity terms for the reader's reference and convenience.

\section{Notations}\label{sec:notation}
We use the following notations for coordinates and differential operator:
\begin{gather*}
	\vec{x} := (x,y,z)^\top, \\
	\vec{x}_h := (x,y)^\top, \\
	\dhori \in \lbrace \partial_x, \partial_y \rbrace, \\
	\nablah : = (\partial_x, \partial_y)^\top, \\
	\Deltah := \partial_{xx} + \partial_{yy}.
\end{gather*}
The non-negative part and the non-positive part of a function $ f $ are defined by
\begin{equation}\label{def:sign-op}
	f^+ := \dfrac{|f| + f}{2} \quad \text{and} \quad f^- := \dfrac{|f| - f}{2}.
\end{equation}
Also, 
\begin{equation*}
	\cb := \cb (\lbrace B_\mrm{j}, \psi_\mrm{j} \rbrace\big|_{\mrm{j}\in \thmoistidx} )
\end{equation*}
is a generic smooth function of $ \lbrace B_\mrm{j}, \psi_\mrm{j} \rbrace\big|_{\mrm{j}\in \thmoistidx} $, defined in \eqref{def:b-psi}, which might be different from line {to} line. {Moreover,}
\begin{equation*}
	\mathcal N(\cdot)
\end{equation*}
will be used to denote a smooth function of the arguments, where subscripts and superscripts might be also applied. 

\section{An approximation system}\label{sec:approximation}
To investigate the well-posedness of system of \eqref{eq:dry-air-mass} -- \eqref{eq:rain-ratio}, we first consider the following approximation system of equations:
\begin{gather}
	 \dt \rhod + \dv (\rhod \vec{u}) =  0, \label{eq:dry-air-mass-app} \\
	 \begin{gathered}\rhod Q_\mathrm{m}^+ [\dt \vec{u} + (\vec{u} \cdot \nabla) \vec{u}]  - \rhod \qr \Vr \dz \vec{u} + \nabla p = \mu \Delta \vec{u} + (\mu + \lambda) \nabla \dv \vec{u} \\ - \rhod Q_\mathrm{m}^+ g \vec{e}_3, \end{gathered} \label{eq:momentum-app} \\
	 \begin{gathered}
		 Q_\mathrm{th}^+ [ \dt \T + (\vec{u}\cdot\nabla) \T ]  - c_\mathrm{l} \qr \Vr \dz \T = \kappa \Delta \T  
		 + Q_\mathrm{cp} \T \dv \vec{u} \\
	  	 - ( Q_1 \T + Q_2 ) ( \Sev^+ - \Scd^+  ) ,\end{gathered} \label{eq:temperature-app} \\
	  \dt \qv + (\vec{u}\cdot\nabla) \qv  = \Sev^+ - \Scd^+ + \Delta \qv, \label{eq:vapor-ratio-app} \\
	  \dt \qc + (\vec{u}\cdot\nabla) \qc  = \Scd^+ - ( \Sac^+ + \Scr^+ ) + \Delta \qc, \label{eq:cloud-ratio-app} \\
	  \dt \qr + (\vec{u}\cdot\nabla) \qr - \dz(\qr \Vr) - \qr \Vr \dz \log \rhod = ( \Sac^+ + \Scr^+ ) - \Sev^+ + \Delta \qr,  \label{eq:rain-ratio-app}
\end{gather}
where the same boundary conditions \eqref{bc:momentum} -- \eqref{bc:moisture} are imposed, and 
\begin{equation}\label{def:state-variable-app}
\begin{aligned}
	Q_\mathrm{m}^+ & := 1 + \qv^+ + \qc^+ + \qr^+, \\
	Q_\mathrm{th}^+ & := \dfrac{c_\mathrm{pd}}{\gamma} + \biggl( \dfrac{c_\mathrm{pv}}{\gamma} + \dfrac{c_\mathrm{pv}}{c_\mathrm{pd}}R_\mathrm{d} - R_\mathrm{v} \biggr) \qv^+  + \biggl( \dfrac{c_\mathrm{l}}{\gamma} + \dfrac{c_\mathrm{l}}{c_\mathrm{pd}}R_\mathrm{d} \biggr)(\qc^+ + \qr^+ ), \\
	\Sev^+ & := c_\mathrm{ev} \T^+ \dfrac{(R_\mathrm{d}+R_\mathrm{v}\qv^+)}{(1+\qv^+ + \qc^+ + \qr^+)} ( q_\mathrm{vs} - \qv^+) ^+ \qr^+ , \\
	\Scd^+ & := c_\mathrm{cd} ( \qv^+ - q_\mathrm{vs} ) \qc^+ + c_\mathrm{cn} ( \qv - q_\mathrm{vs} )^+ q_\mathrm{cn}, \\
	\Sac^+ & :=  c_\mathrm{ac} ( \qc - q_\mathrm{ac})^+,  \\
	\Scr^+ & :=  c_\mathrm{cr} \qc^+ \qr^+.
\end{aligned}
\end{equation}

\begin{remark}
	In comparison to previous study in {\cite{Hittmeir2017, Hittmeir2019}}, the evaporation source term does not contain the general exponent $\beta\in(0,1]$ as exponent of $q_r$ anymore. Here only the case corresponding to $\beta=1$ is treated (see also in \cite{Hittmeir2023}). Then in particular,
	\begin{equation*}
		|S_\mrm j ^+| \lesssim 1 + |\T, \qv, \qc, \qr |^2, ~~ \mrm j \in \lbrace \mrm{ev}, \mrm{cd}, \mrm{ac}, \mrm{cr} \rbrace.
	\end{equation*}
\end{remark}

To get rid of the inhomogeneity of the boundary conditions \eqref{bc:temperature} and \eqref{bc:moisture}, we introduce a set of equivalent unknowns in the following. 
Notice that, \eqref{bc:temperature} and \eqref{bc:moisture} are in the form of 
\begin{equation}\label{bc:abstract}
		\partial_z \mathcal{F}\big|_{z=0,1} = \alpha_\mathcal{F}^b( \mathcal F^b - \mathcal{F} )\big|_{z=0,1}, ~~~~ \mathcal F \in \lbrace \mathcal T, \qv, \qc, \qr \rbrace.  
		\end{equation}
		We remind the reader that $ \alpha_{\mathcal F}^b $ takes different constants for $ z = 0, 1 $. Since $ \alpha_\mathcal{F}^b \big|_{z=0,1} = \bigl( \alpha_\mathcal{F}^b \big|_{z=0} (1-z) + \alpha_\mathcal{F}^b \big|_{z=1} z \bigr)\big|_{z=0,1} $, one can rewrite \eqref{bc:abstract} as
		$$
		\dz ( e^{A_\mathcal{F}(z)} \mathcal F)\big|_{z=0,1} = \alpha_\mathcal{F}^b e^{A_\mathcal{F}(z)} \mathcal F^b \big|_{z=0,1},
		$$
		where
		$$
		A_\mathcal{F}(z) := - \dfrac{\alpha_\mathcal{F}^b\big|_{z=0}}{2} (1-z)^2 +  \dfrac{\alpha_\mathcal{F}^b\big|_{z=1}}{2} z^2.  
		$$
		In addition, there exists $ \psi_\mathcal{F}:= \Psi( \alpha_\mathcal{F}^b e^{A_\mathcal{F}(z)} \mathcal F^b \big|_{z=0,1} ) $ (see extension Lemma \ref{lm:extension} in Appendix), such that $ \dz \psi_\mathcal{F}\big|_{z=0,1} = \alpha_\mathcal{F}^b e^{A_\mathcal{F}(z)} \mathcal F^b \big|_{z=0,1} $, and for any $ s \in [0,\infty) $, 
		\begin{equation}\label{est:trace-ori}
		\norm{\psi_\mathcal{F}}{\Hnorm{s+3/2}} \leq C_{s} \norm{\mathcal{F}^b}{\bHnorm{s}} .
		\end{equation}
		Here $ \Gamma := 2\mathbb T^2 \times \lbrace z = 0,1\rbrace = \partial \Omega $ is the boundary of the domain $ \Omega $.
		With such notations, 
one can define
\begin{equation}\label{def:homo-variable}
		\mathfrak T := B_\mathcal{T} \mathcal{T} - \psi_\mathcal{T}, ~~
		\mathfrak{q}_{\mrm{j}} := B_{\mrm{j}} q_\mrm{j} - \psi_{\mrm{j}}, ~~~~ \mrm{j} \in \lbrace \mrm{v}, \mrm{c}, \mrm{r} \rbrace,
\end{equation}
with
\begin{equation}\label{def:b-psi}
\begin{gathered}
	B_\mathcal{T} :=   e^{A_\mathcal{T}(z)}, ~~~~ A_\mathcal{T}(z) :=  - \dfrac{\alpha_\mathcal{T}^b\big|_{z=0}}{2} (1-z)^2 +  \dfrac{\alpha_\mathcal{T}^b\big|_{z=1}}{2} z^2, \\
	\psi_\mathcal T :=  \Psi(\alpha_\mathcal{T}^b e^{A_\mathcal{T}(z)}\mathcal{T}^b\big|_{z=0,1}) ,   \\
	B_\mrm{j} :=  e^{A_{\mrm{j}}(z)}, ~~~~ A_{\mrm{j}}(z) :=  - \dfrac{\alpha_\mrm{j}^b\big|_{z=0}}{2} (1-z)^2 +  \dfrac{\alpha_\mrm{j}^b\big|_{z=1}}{2} z^2, \\
	\psi_\mrm{j} :=  \Psi(\alpha_{\mrm{j}}^b e^{A_{\mrm{j}}(z)}q_\mrm{j}^b\big|_{z=0,1}), ~~~~ \mrm{j} \in \lbrace  \mrm{v}, \mrm{c}, \mrm{r} \rbrace.
\end{gathered}
\end{equation}
Notice that, $ B_\mrm{j}, \psi_\mrm{j} $, $ \mrm{j} \in \lbrace \mathcal{T}, \mrm{v}, \mrm{c}, \mrm{r} \rbrace $, are independent of the solutions. To simplify the presentation, we assume that $ \lbrace q_\mrm{j}^b\big|_{z=0,1} \rbrace_{\mrm{j}\in \moistidx} $ are given smooth functions, i.e., $ C^\infty $ in space and time, so that $ \lbrace B_\mrm{j}, \psi_\mrm{j} \rbrace\big|_{z=0,1} $ are smooth.

Then $ \mathfrak{T}, \mathfrak{q}_\mrm{j} $, $ \mrm{j} \in \lbrace \mrm{v}, \mrm{c}, \mrm{r} \rbrace $, satisfy, after substituting \eqref{eq:temperature-app} -- \eqref{eq:rain-ratio-app} and \eqref{bc:temperature} -- \eqref{bc:moisture}, the following system of equations and boundary conditions:
\begin{gather}
	\begin{gathered} 
		Q_\mrm{th}^+ \dt \mathfrak T - \kappa \Delta \mathfrak T = - Q_\mrm{th}^+ \vec{u} \cdot\nabla \mathfrak T - c_\mrm{l}B_\mrm{r}^{-1} \Vr (\mathfrak q_\mrm{r} + \psi_\mrm{r}) \dz \mathfrak T \\
		- 2 \kappa \dz \log B_\T \dz \mathfrak T 
		+ Q_\mrm{th}^+ \mathfrak T w \dz \log B_\T + c_\mrm{l} B_\mrm{r}^{-1}\Vr (\mathfrak q_\mrm{r} + \psi_\mrm{r}) \mathfrak T \dz \log B_\T \\ + \kappa \partial_{zz} B_\T^{-1} B_\T \mathfrak T  
		+ Q_\mrm{cp} \mathfrak T \dv \vec{u} - Q_1 \mathfrak T (\Sev^+ - \Scd^+ ) \\
		- Q_\mrm{th}^+ \vec{u} \cdot\nabla \psi_\T  
		- c_\mrm{l}B_\mrm{r}^{-1} \Vr (\mathfrak q_\mrm{r} + \psi_\mrm{r}) \dz \psi_\T
		+ Q_\mrm{th}^+ \psi_\T w \dz \log B_\T \\
		+ c_\mrm{l} B_\mrm{r}^{-1}\Vr (\mathfrak q_\mrm{r}+ \psi_\mrm{r}) \psi_\T \dz \log B_\T
		 - Q_\mrm{th}^+ \dt \psi_\T + \kappa \Delta \psi_\T \\
		 - 2 \kappa \dz \log B_\T \dz \psi_\T + \kappa \partial_{zz} B_\T^{-1} B_\T \psi_\T + Q_\mrm{cp} \psi_\T \dv \vec{u} \\ - (Q_1 \psi_\T + Q_2 B_\T  ) (\Sev^+ - \Scd^+ ),
	\end{gathered} \label{eq:temperature-app-1} \\
	\begin{gathered}
	\dt \mathfrak q_\mrm{v} - \Delta \mathfrak q_\mrm{v} = - \vec{u} \cdot \nabla \mathfrak q_\mrm{v}  - 2 \dz \log B_\mrm{v} \dz \mathfrak q_\mrm{v} 
	+ \mathfrak q_\mrm{v} w \dz \log B_\mrm{v} \\
	 + \partial_{zz} B_\mrm{v}^{-1} B_\mrm{v} \mathfrak q_\mrm{v}  
	+ ( \Sev^+ - \Scd^+) B_\mrm{v}
	- \vec{u} \cdot \nabla \psi_\mrm{v} \\ + \psi_\mrm{v} w \dz \log B_\mrm{v}
	- \dt \psi_\mrm{v} + \Delta \psi_\mrm{v} - 2 \dz \log B_\mrm{v} \dz \psi_\mrm{v} \\
	  + \partial_{zz} B_\mrm{v}^{-1} B_\mrm{v} \psi_\mrm{v},
	\end{gathered}	\label{eq:vapor-ratio-app-1} \\
		\begin{gathered}
		\dt \mathfrak q_\mrm{c} - \Delta \mathfrak q_\mrm{c} = - \vec{u} \cdot \nabla \mathfrak q_\mrm{c}  - 2 \dz \log B_\mrm{c} \dz \mathfrak q_\mrm{c} 
		+ \mathfrak q_\mrm{c} w \dz \log B_\mrm{c} \\
		+ \partial_{zz} B_\mrm{c}^{-1} B_\mrm{c} \mathfrak q_\mrm{c}  
		+ (\Scd^+ - \Sac^+ - \Scr^+) B_\mrm{c}
		- \vec{u} \cdot \nabla \psi_\mrm{c} \\ + \psi_\mrm{c} w \dz \log B_\mrm{c}
		- \dt \psi_\mrm{c} + \Delta \psi_\mrm{c} - 2 \dz \log B_\mrm{c} \dz \psi_\mrm{c} \\
		+ \partial_{zz} B_\mrm{c}^{-1} B_\mrm{c} \psi_\mrm{c},
		\end{gathered} \label{eq:cloud-ratio-app-1}\\
	\begin{gathered}
	\dt \mathfrak q_\mrm{r} - \Delta \mathfrak q_\mrm{r} = - \vec{u} \cdot \nabla \mathfrak q_\mrm{r}  + \Vr \dz \mathfrak q_\mrm{r} - 2 \dz \log B_\mrm{r} \dz \mathfrak q_\mrm{r} 
	+ \mathfrak q_\mrm{r} w \dz \log B_\mrm{r} \\
	+ \partial_{zz} B_\mrm{r}^{-1} B_\mrm{r} \mathfrak q_\mrm{r} 
	+ \mathfrak{q}_{r} ( \dz \Vr +\Vr \dz \log \rhod - \Vr \dz \log B_\mrm{r} ) \\
	+ (\Sac^+ + \Scr^+ - \Sev^+) B_\mrm{r}
	- \vec{u} \cdot \nabla \psi_\mrm{r} + \Vr \dz \psi_\mrm{r} \\ + \psi_\mrm{r} w \dz \log B_\mrm{r}
	- \dt \psi_\mrm{r} + \Delta \psi_\mrm{r} - 2 \dz \log B_\mrm{r} \dz \psi_\mrm{r} \\
	+ \partial_{zz} B_\mrm{r}^{-1} B_\mrm{r} \psi_\mrm{r} + \psi_\mrm{r} ( \dz \Vr +\Vr \dz \log \rhod - \Vr \dz \log B_\mrm{r} ),
	\end{gathered} \label{eq:rain-ratio-app-1}
\end{gather}
with
\begin{equation}\label{bc:app-hom}
	\dz \mathfrak T\big|_{z=0,1} = \dz \mathfrak q_\mrm{j}\big|_{z=0,1} = 0, ~~~~ \mrm{j} \in \lbrace \mrm{v}, \mrm{c}, \mrm{r} \rbrace.
\end{equation}

In the next section, we will study a linear inhomogeneous system of equations associated with \eqref{eq:dry-air-mass-app}, \eqref{eq:momentum-app}, and \eqref{eq:temperature-app-1} -- \eqref{eq:rain-ratio-app-1}, with homogeneous boundary conditions \eqref{bc:momentum} and \eqref{bc:app-hom}. Our strategy is to construct the solutions to the nonlinear problem by a fixed point argument. 

\section{The associated linear system and the corresponding {contraction} mapping}\label{sec:linear_contract}

Let $ (\vec{u}^o , \mathfrak T^o, \lbrace\mathfrak q_\mrm{j}^o\rbrace_{\mrm{j} \in \lbrace \mrm{v}, \mrm{c}, \mrm{r} \rbrace}) $ be an element of $ \mathfrak{Y} $, defined in \eqref{def:compact-subset}, below. We now introduce the associated linear approximations to system \eqref{eq:dry-air-mass-app}, \eqref{eq:momentum-app}, \eqref{eq:temperature-app-1} -- \eqref{eq:rain-ratio-app-1}. 
Let $ \rhod $ be the solution of
\begin{equation}\label{eq:dry-air-mass-lin}
	\dt \rhod + \dv ( \rhod \vec{u}^o ) = 0,
\end{equation}
which can be solved by the standard characteristic method.
Consider the following linear system of  equations for $ \vec{u}, \mathfrak T, \mathfrak q_\mrm{j} $, $ \mrm{j} \in \lbrace \mrm{v}, \mrm{c}, \mrm{r} \rbrace $:
\begin{gather}
	\rhod Q_\mathrm{m}^{+,o} \dt \vec{u} - \mu \Delta \vec{u} - (\mu + \lambda) \nabla \dv \vec{u}  = - \nabla p^o +  \mathcal I_{\vec{u},1} + \mathcal{I}_{\vec{u},2}, \label{eq:momentum-lin} \\
		 Q_\mathrm{th}^{+,o} \dt \mathfrak T  - \kappa \Delta \mathfrak T = \mathcal I_{\T,1} + \mathcal I_{\T,2}, 
		 \label{eq:temperature-lin} \\
	  \dt \mathfrak q_\mrm{v} - \Delta\mathfrak q_\mrm{v} = \mathcal I_{\mrm{v},1} + \mathcal I_{\mrm{v},2}, \label{eq:vapor-ratio-lin} \\
	  \dt \mathfrak q_\mrm{c} - \Delta\mathfrak q_\mrm{c} = \mathcal I_{\mrm{c},1} + \mathcal I_{\mrm{c},2}, \label{eq:cloud-ratio-lin} \\
	  \dt \mathfrak q_\mrm{r} - \Delta\mathfrak q_\mrm{r} = \mathcal I_{\mrm{r},1} + \mathcal I_{\mrm{r},2},  \label{eq:rain-ratio-lin}
\end{gather}
with homogeneous boundary conditions
\begin{equation}\label{bc:app-lin}
		\begin{gathered}
		\dz v\big|_{z= 0 ,1} = 0, ~ w \big|_{z=0,1} = 0, 	\\
		\dz \mathfrak T\big|_{z=0,1} = \dz \mathfrak q_\mrm{j}\big|_{z=0,1} = 0, ~~~~ \mrm{j} \in \lbrace \mrm{v}, \mrm{c}, \mrm{r} \rbrace,
		\end{gathered}
\end{equation}
where
\begin{align}
	& \begin{aligned}
	\mathcal{I}_{\vec{u}, 1} := & - \rhod Q_\mrm{m}^{+,o} (\vec{u}^o \cdot \nabla) \vec{u}^o + \rhod B_\mrm{r}^{-1} (\mathfrak q_\mrm{r}^o + \psi_\mrm{r}) \Vr \dz \vec{u}^o ,  \\
	\mathcal{I}_{\vec{u}, 2} := &  
	- \rhod Q_\mathrm{m}^{+,o} g \vec{e}_3 , 
	\end{aligned} \label{def:momentum-source-lin} \\
	& \begin{aligned}
		\mathcal{I}_{\T, 1} := & - Q_\mrm{th}^{+,o} \vec{u}^o \cdot\nabla \mathfrak T^o - c_\mrm{l}B_\mrm{r}^{-1} \Vr (\mathfrak q_\mrm{r}^o + \psi_\mrm{r}) \dz \mathfrak T^o 
		- 2 \kappa \dz \log B_\T \dz \mathfrak T^o \\
		&  + Q_\mrm{cp}^o \mathfrak T^o \dv \vec{u}^o + Q_\mrm{cp}^o \psi_\T \dv \vec{u}^o, \\
		\mathcal{I}_{\T, 2} := & 
		\mathcal I_{\T, 2}' 
		 - Q_1 \mathfrak T^o (\Sev^{+,o} - \Scd^{+,o} ) - (Q_1 \psi_\T + Q_2 B_\T  ) (\Sev^{+,o} - \Scd^{+,o} ),
	\end{aligned} \label{def:temperature-source-lin} \\
	& \begin{aligned}
		\mathcal{I}_{\mrm{v}, 1} := & - \vec{u}^o \cdot \nabla \mathfrak q_\mrm{v}^o  - 2 \dz \log B_\mrm{v} \dz \mathfrak q_\mrm{v}^o, \\
		\mathcal{I}_{\mrm{v}, 2} := &  \mathcal I_{\mrm v, 2}' 
		+ ( \Sev^{+,o} - \Scd^{+,o}) B_\mrm{v}
		,
	\end{aligned} \label{def:vapor-source-lin} \\
	& \begin{aligned}
		\mathcal{I}_{\mrm{c}, 1} := &  - \vec{u}^o \cdot \nabla \mathfrak q_\mrm{c}^o  - 2 \dz \log B_\mrm{c} \dz \mathfrak q_\mrm{c}^o , \\
		\mathcal{I}_{\mrm{c}, 2} := &  \mathcal I_{\mrm c, 2}' 
		+ (\Scd^{+,o} - \Sac^{+,o} - \Scr^{+,o}) B_\mrm{c},
	\end{aligned} \label{def:cloud-source-lin} \\
	& \begin{aligned}
		\mathcal{I}_{\mrm{r}, 1} := & - \vec{u}^o \cdot \nabla \mathfrak q_\mrm{r}^o  + \Vr \dz \mathfrak q_\mrm{r}^o - 2 \dz \log B_\mrm{r} \dz \mathfrak q_\mrm{r}^o , \\
		\mathcal{I}_{\mrm{r}, 2} := & \mathcal I_{\mrm r, 2}'
		+ (\mathfrak{q}_{r}^o + \psi_\mrm{r}) ( \dz \Vr +\Vr \dz \log \rhod - \Vr \dz \log B_\mrm{r} ) \\&
		+ \Vr \dz \psi_\mrm{r} + (\Sac^{+,o} + \Scr^{+,o} - \Sev^{+,o}) B_\mrm{r},
	\end{aligned} \label{def:rain-source-lin} 
\end{align}
and, for $ \mrm{j} \in \moistidx $, 
\begin{align*}
&\begin{aligned}
 \mathcal I_{\T, 2}' :=& 
Q_\mrm{th}^{+,o} (\mathfrak T^o + \psi_\T) w^o \dz \log B_\T + \kappa \partial_{zz} B_\T^{-1} B_\T ( \mathfrak T^o + \psi_\T)  \\
&  + c_\mrm{l} B_\mrm{r}^{-1}\Vr (\mathfrak q_\mrm{r}^o + \psi_\mrm{r}) (\mathfrak T^o + \psi_\T) \dz \log B_\T \\ & 
- Q_\mrm{th}^{+,o} \vec{u}^o \cdot\nabla \psi_\T  
- c_\mrm{l}B_\mrm{r}^{-1} \Vr (\mathfrak q_\mrm{r}^o + \psi_\mrm{r}) \dz \psi_\T \\&
- Q_\mrm{th}^{+,o} \dt \psi_\T + \kappa \Delta \psi_\T  
- 2 \kappa \dz \log B_\T \dz \psi_\T ,
\end{aligned} \\
	&\begin{aligned}
	& \mathcal I_{\mrm{j},2}' := 
	(\mathfrak q_\mrm{j}^o + \psi_\mrm{j}) w^o \dz \log B_\mrm{j} 
	+ \partial_{zz} B_\mrm{j}^{-1} B_\mrm{j} (\mathfrak q_\mrm{j}^o + \psi_\mrm{j} ) \\& ~~~~
	- \vec{u}^o \cdot \nabla \psi_\mrm{j}  
	- \dt \psi_\mrm{j} + \Delta \psi_\mrm{j} - 2 \dz \log B_\mrm{j} \dz \psi_\mrm{j}.
	\end{aligned}
\end{align*}
Here,
\begin{equation}\label{def:state-variable-lin}
	\begin{aligned}
	Q_\mathrm{m}^{o,+} := & 1 + B_\mrm{v}^{-1} (\mathfrak q_\mrm{v}^o + \psi_\mrm{v})^+  + B_\mrm{c}^{-1} (\mathfrak q_\mrm{c}^o + \psi_\mrm{c})^+ + B_\mrm{r}^{-1} (\mathfrak q_\mrm{r}^o + \psi_\mrm{r})^+, \\
	Q_\mathrm{th}^{o,+} := & \dfrac{c_\mathrm{pd}}{\gamma} + \biggl( \dfrac{c_\mathrm{pv}}{\gamma} + \dfrac{c_\mathrm{pv}}{c_\mathrm{pd}}R_\mathrm{d} - R_\mathrm{v} \biggr) B_\mrm{v}^{-1} (\mathfrak q_\mrm{v}^o + \psi_\mrm{v})^+ \\
	& ~~~~ + \biggl( \dfrac{c_\mathrm{l}}{\gamma} + \dfrac{c_\mathrm{l}}{c_\mathrm{pd}}R_\mathrm{d} \biggr)(B_\mrm{c}^{-1} (\mathfrak q_\mrm{c}^o + \psi_\mrm{c})^+ + B_\mrm{r}^{-1} (\mathfrak q_\mrm{r}^o + \psi_\mrm{r})^+ ), \\
		Q_\mathrm{cp}^{o} := &  - R_\mathrm{d} - R_\mathrm{v} B_\mrm{v}^{-1} (\mathfrak q_\mrm{v}^o + \psi_\mrm{v}), \\
		p^o := & \rhod (R_\mrm{d} + R_\mrm{v} B_\mrm{v}^{-1} (\mathfrak q_\mrm{v}^o + \psi_\mrm{v}) ) B_\T^{-1} (\mathfrak T^o + \psi_\T), \\
	\Sev^{+,o}  := &  c_\mathrm{ev} B_\T^{-1} (\mathfrak T^o + \psi_\T)^+ ( q_\mathrm{vs}^o - B_\mrm{v}^{-1} (\mathfrak q_\mrm{v}^o + \psi_\mrm{v})^+) ^+ B_\mrm{r}^{-1} (\mathfrak q_\mrm{r}^o + \psi_\mrm{r})^+ \\
	&  \cdot \dfrac{(R_\mathrm{d}+R_\mathrm{v} B_\mrm{v}^{-1} (\mathfrak q_\mrm{v}^o + \psi_\mrm{v})^+)}{(1+B_\mrm{v}^{-1} (\mathfrak q_\mrm{v}^o + \psi_\mrm{v})^+ + B_\mrm{c}^{-1} (\mathfrak q_\mrm{c}^o + \psi_\mrm{c})^+ + B_\mrm{r}^{-1} (\mathfrak q_\mrm{r}^o + \psi_\mrm{r})^+)}  , \\
	\Scd^{+,o} := & c_\mathrm{cd} ( B_\mrm{v}^{-1} (\mathfrak q_\mrm{v}^o + \psi_\mrm{v})^+ - q_\mathrm{vs}^o ) B_\mrm{c}^{-1} (\mathfrak q_\mrm{c}^o + \psi_\mrm{c})^+ \\
	& ~~~~ + c_\mathrm{cn} ( B_\mrm{v}^{-1} (\mathfrak q_\mrm{v}^o + \psi_\mrm{v}) - q_\mathrm{vs}^o )^+ q_\mathrm{cn}, \\
	\Sac^{+,o} := &  c_\mathrm{ac} ( B_\mrm{c}^{-1} (\mathfrak q_\mrm{c}^o + \psi_\mrm{c}) - q_\mathrm{ac})^+,  \\
	\Scr^{+,o} := &  c_\mathrm{cr} B_\mrm{c}^{-1} (\mathfrak q_\mrm{c}^o + \psi_\mrm{c})^+ B_\mrm{r}^{-1} (\mathfrak q_\mrm{r}^o + \psi_\mrm{r})^+, \\
	q_\mrm{vs}^o := & q_\mrm{vs}(p^o, B_\T^{-1} (\mathfrak T^o + \psi_\T)).
	\end{aligned}
\end{equation}


To introduce our contracting mapping scheme, for some finite time $ T^* \in (0,\infty) $, to be determined later, 
we set $$ \mathfrak M_{T^*} := L^\infty(0,T^*;L^2(\Omega))\cap L^2(0,T^*;H^1(\Omega)). $$
{The next steps are introduced to establish the local existence theory}:
Step 1 {is to show} well-posedness of the associated linear system of equations; Step 2 is then  to  derive bounds for the solutions in certain compact subspace $ \mathfrak X_{T^*} $ of $ \mathfrak M_{T^*} $, which consequently imply that the contraction mapping $ \mathcal M $ (see \eqref{def:mapping}, below) is well-defined; Step 3 {consists of proving the contraction property of $ \mathcal M $ in $ \mathfrak M_{T^*} $}.

Consider the initial data $$ (\rho_{\mrm d, \mrm{in}}, \mathfrak T_{\mrm{in}} ,  \lbrace \mathfrak q_{\mrm j, \mrm{in}}\rbrace_{\mrm j \in \moistidx}, \vec u_{\mrm{in}})\, , $$ for $ (\rhod, \mathfrak T, \lbrace \mathfrak q_{\mrm j}\rbrace_{\mrm j \in \moistidx}, \vec u) $, to be smooth enough as in Theorem \ref{thm}. Let $ c_\mrm{in,o} \in (0,\infty) $ be the bounds of initial data defined by 
\begin{equation}\label{def:initial-bound}
	\begin{gathered}
	\norm{\rho_{\mrm d, \mrm{in}}^{1/2}, \log \rho_{\mrm d, \mrm{in}} , \vec u_\mrm{in} , \mathfrak T_\mrm{in} , \lbrace \mathfrak q_{\mrm j, \mrm{in}} \rbrace_{\mrm j \in \moistidx} }{\Hnorm{2}}^2 \leq \mathfrak c_\mrm{in,o}.
	\end{gathered}
\end{equation}
Notice, \eqref{def:initial-bound} implies that $ \rho_{\mrm d, \mrm{in}} $ is strictly positive, and thus from \eqref{eq:momentum-lin} -- \eqref{eq:rain-ratio-lin}, one can check that
\begin{equation} \label{def:initial-bound-2}
	\norm{\dt \vec u \big|_{t=0}, \dt \mathfrak T\big|_{t=0}, \lbrace \dt \mathfrak q_{\mrm j}\big|_{t=0} \rbrace_{\mrm j \in \moistidx}}{\Lnorm{2}}^2 < \mathcal N_{\mrm{in}}(\mathfrak c_{\mrm{in},o}) < \infty,
\end{equation} 
for some smooth function $ N_\mrm{in}(\cdot) $. Here,  $$ (\dt \vec u \big|_{t=0}, \dt \mathfrak T\big|_{t=0}, \lbrace \dt \mathfrak q_{\mrm j}\big|_{t=0} \rbrace_{\mrm j \in \moistidx})  $$
are the initial {data} for the time derivatives defined by equations \eqref{eq:momentum-app} and \eqref{eq:temperature-app-1} -- \eqref{eq:rain-ratio-app-1}. 
Let $ \mathfrak c_{\mrm{in}} = \max \lbrace \mathfrak c_{\mrm{in},o}, \mathcal N_{\mrm{in}}(\mathfrak c_{\mrm{in},o}) \rbrace $. 

We introduce the aforementioned compact (see justification, below) subspace $\mathfrak X_{T^*}$ of $  \mathfrak M_{T^*} $  for our solutions by:
\begin{equation}\label{def:compact-space}
\begin{gathered}
 \mathfrak X_{T^*} = \bigl\lbrace (\vec u, \mathfrak T, \lbrace \mathfrak q_\mrm j \rbrace_{\mrm j \in \moistidx})   \in L^\infty(0, T^*; H^2(\Omega))\cap L^2(0, T^*;H^3(\Omega ))| \\ 
 \dt \mathfrak T, \lbrace \dt \mathfrak q_\mrm j \rbrace_{\mrm j \in \moistidx} \in L^\infty(0, T^*; L^2(\Omega))\cap L^2(0, T^*;H^1(\Omega )),\\
  \dt  \vec u \in L^2 (0,T^*;H^1(\Omega)),\\
  (\vec u, \mathfrak T, \lbrace \mathfrak q_\mrm j \rbrace_{\mrm j \in \moistidx})\big|_{t=0} = (\vec u_\mrm{in}, \mathfrak T_\mrm{in}, \lbrace \mathfrak q_{\mrm j, \mrm{in}} \rbrace_{\mrm j \in \moistidx})
 \bigr\rbrace.
\end{gathered}
\end{equation}
In addition, for some constants, $ \mathfrak c_{u,1}, \mathfrak c_{u,2}, \mathfrak c_{u,3}, \mathfrak c_{o,1}, \mathfrak c_T, \mathfrak c_q, \mathfrak c_{o,2} $, to be determined later, consider the following inequalities: for $ t \in (0,T^*] $, 
\begin{gather}
	\sup_{0\leq s \leq t}\norm{\vec u(s)}{\Hnorm{1}}^2 + \int_0^t \norm{\dt \vec u(s)}{\Hnorm{1}}^2 \,dt \leq \mathfrak c_{u,1}, \label{unest:u-1}
	\\
	\sup_{0\leq s \leq t} \norm{\nabla^2 \vec u (s)}{\Lnorm{2}}^2 \leq \mathfrak c_{u,2}, \label{unest:u-2}
	\\
	\int_0^t \norm{\nabla^3 \vec u(s)}{\Lnorm{2}}^2 \leq \mathfrak c_{u,3}, \label{unest:u-3}
	\\
\begin{aligned}
	& \sup_{0\leq s \leq t} \bigl( \norm{\dt  \mathfrak T(s), \lbrace \dt \mathfrak q_\mrm{j}(s) \rbrace_{\mrm j \in \moistidx} }{\Lnorm{2}}^2   + \norm{\mathfrak T(s), \lbrace \mathfrak q_\mrm{j}(s)\rbrace_{\mrm j \in \moistidx}}{\Hnorm{1}}^2 \bigr)\\
	& ~~~~ ~~~~ + \int_0^t \norm{\dt  \mathfrak T(s), \lbrace \dt \mathfrak q_\mrm{j}(s) \rbrace_{\mrm j \in \moistidx} }{\Hnorm{1}}^2 \,ds \leq \mathfrak c_{o,1},
\end{aligned} \label{unest:o-1}	\\
	\sup_{0\leq s \leq t}\norm{\nabla^2 \mathfrak T}{\Lnorm{2}}^2 \leq \mathfrak c_{T}, \label{unest:T}
	\\
	\sup_{0\leq s \leq t}\norm{\lbrace\nabla^2 \mathfrak q_\mrm{j}\rbrace_{\mrm  j \in \moistidx} }{\Lnorm{2}}^2 \leq \mathfrak c_{q},   \label{unest:q}
	\\
	\int_0^t \norm{\nabla^3 \mathfrak T, \lbrace \nabla^3 \mathfrak q_\mrm{j} \rbrace_{\mrm j \in \moistidx}}{\Lnorm{2}}^2\leq \mathfrak c_{o,2}. \label{unest:o-2}
\end{gather}

Also, let $ \mathfrak Y $ be the bounded subset of $ \mathfrak X_{T^*} $, admitting \eqref{unest:u-1} -- \eqref{unest:o-2}, i.e.,
\begin{equation}\label{def:compact-subset}
\begin{gathered}
\mathfrak Y := \bigl\lbrace (\vec u, \mathfrak T, \lbrace \mathfrak q_\mrm j \rbrace_{\mrm j \in \moistidx})   \in \mathfrak X_{T^*}| ~ \text{Inequalities \eqref{unest:u-1} -- \eqref{unest:o-2} hold.} \bigr\rbrace.
\end{gathered}
\end{equation}
We remark that $ \mathfrak Y $ is a subset of $ \mathfrak X_{T^*} $, and thanks to the Aubin-Lions Lemma (e.g., \cite[Theorem 2.1]{temam1977} and \cite{Simon1986}), $ \mathfrak X_{T^*} $ is compactly embedded in $ \mathfrak M_{T^*} $, i.e., 
$$
\mathfrak Y \subset \mathfrak X_{T^*} \subset\subset \mathfrak M_{T^*} . 
$$

\subsection{Well-posedness of the associated linear system \eqref{eq:momentum-lin} -- \eqref{eq:rain-ratio-lin}}

By assuming $ \vec u^o, \mathfrak T^o $, and $ \lbrace \mathfrak q_\mrm j^o \rbrace_{\mrm j \in \moistidx} $ to be smooth enough, \eqref{eq:dry-air-mass-lin} can be solved by the standard characteristic method, and the well-posedness of \eqref{eq:momentum-lin} -- \eqref{eq:rain-ratio-lin} follows from a standard Galerkin method. Such a Galerkin scheme is similar to the one employed in \cite{LT2021-CPE}.

To be more precise, let $ \lbrace \Psi_\mrm j \rbrace_{\mrm j = 1,2,3,\cdots} $ and $ \lbrace \Phi_\mrm j\rbrace_{\mrm j = 1,2,3,\cdots} $ be the eigenfunctions of the Laplacian subject to the Dirichlet and the Neumann boundary conditions, respectively, i.e., for $ \mrm j = 1,2,3,\cdots $, 
$$
\begin{aligned}
	- \Delta \Psi_\mrm j & = \sigma_j \Psi_\mrm j, \qquad & \Psi_\mrm j\vert_{z=0,1} & = 0, \\
	- \Delta \Phi_\mrm j & = \lambda_j \Phi_\mrm j, \qquad & \partial_z \Phi_\mrm j\vert_{z=0,1} & = 0, \qquad \int \Phi_\mrm j \idx = 1,
\end{aligned}
$$
where $ \lambda_\mrm j$ and $ \sigma_\mrm j $ are the corresponding eigenvalues, labeled in the increasing order. Then the Galerkin approximation scheme can be described as follows: Let $ K = 1,2,3,\cdots $ be an arbitrary finite integer. Let $ v = v^K, \mathfrak T = \mathfrak T^K, \lbrace \mathfrak q_\mrm j = \mathfrak q_\mrm j^K \rbrace_{\mrm j \in \moistidx} $ be in the finite dimensional vector space spanned by $ \lbrace \Phi_j, \mrm j = 1,2,3,\cdots, K \rbrace $ and $ w = w^K $ be in the finite dimensional vector space spanned by $ \lbrace \Psi_\mrm j, \mrm j = 1,2,3,\cdots, K \rbrace $. Then by restricting \eqref{eq:momentum-lin}--\eqref{eq:rain-ratio-lin} in the corresponding finite dimensional space, it forms a system of linear, non-degenerate ordinary differential equations (ODEs) with source terms. Then the existence of local-in-time solutions for this system of ODEs follows from the standard existence theory. Moreover, the standard $ L^2 $-estimates implies the uniform-in-$ K $ bounds of $ v^K, w^K, \mathfrak T^K $, and $ \lbrace \mathfrak q_\mrm j^K \rbrace_{\mrm j \in \moistidx} $. After sending $ K \rightarrow \infty $, one can conclude the existence of solutions to \eqref{eq:momentum-lin} -- \eqref{eq:rain-ratio-lin}. The details are omitted here.



\subsection{Definition of the contract mapping}
In this subsection, we aim at defining the contract mapping which maps $ (\vec u^o, \mathfrak T^o, \lbrace \mathfrak q_\mrm j^o \rbrace_{\mrm j \in \moistidx}) $ to the solution of \eqref{eq:dry-air-mass-lin}--\eqref{eq:rain-ratio-lin}. This is done by first 
establishing the estimates \eqref{unest:u-1} -- \eqref{unest:o-2}, below. We start with the estimates of $ \rhod $. 

\subsubsection*{The $ \rhod $-estimates} 
One can write, following \eqref{eq:dry-air-mass-lin}, 
\begin{gather}
	\dt \rhod^{1/2} + \vec u^o \cdot \nabla \rhod^{1/2} + \dfrac{1}{2} \rhod^{1/2} \dv \vec u^o = 0, \label{eq:half-density-lin} \\
	\dt \log \rhod + \vec u^o \cdot \nabla \log \rhod + \dv \vec{u}^o = 0. \label{eq:log-density-lin}
\end{gather}
Applying standard $ H^2 $-estimates yields that,
\begin{gather*}
	\dfrac{d}{dt} \norm{ \rhod^{1/2}}{\Hnorm{2}}^2 \lesssim \norm{\nabla \vec{u}^o}{\Hnorm{2}} \norm{\rhod^{1/2}}{\Hnorm{2}}^2, \\
	\dfrac{d}{dt} \norm{\log \rhod}{\Hnorm{2}}^2 \lesssim  \norm{\nabla \vec{u}^o}{\Hnorm{2}} (\norm{\log \rhod}{\Hnorm{2}}^2 + 1 ).
\end{gather*}
Then, applying Gr\"onwall's inequality to the above inequalities leads to the estimates: 
\begin{equation}\label{linest:d-1}
	\begin{split}
	&\sup_{0\leq s\leq  t} \norm{\rhod^{1/2}(s), \log \rhod(s) }{\Hnorm{2}}^2\\
 &\quad {\leq c_{d,1} e^{c_{d,1} \int_0^t \norm{\nabla \vec u^o(s)}{\Hnorm{2}} \,ds } 
	 \Bigl(\int_0^t (\norm{\nabla \vec u^o(s)}{\Hnorm{2}} + 1) \,ds +C_{id} \Bigr),}
	\end{split}
\end{equation}
for some positive constant $ c_{d,1} \in (0,\infty)$, {where $C_{id}$ is an upper bound for the terms involving the initial data}.
Meanwhile, directly from \eqref{eq:half-density-lin} and \eqref{eq:log-density-lin}, one can derive
\begin{equation}\label{linest:d-2}
	\norm{\dt \rhod^{1/2}, \dt \log \rhod}{\Hnorm{1}}^2 \leq  c_{d,2}\norm{\vec{u}^o}{\Hnorm{2}}^2 ( \norm{\rhod^{1/2},\log \rhod}{\Hnorm{2}}^2 + 1),
\end{equation}
for some positive constant $ c_{d,2} \in (0,\infty) $. 

\subsubsection*{The $ \vec{u} $-estimates}

Taking the $ L^2 $-inner product of \eqref{eq:momentum-lin} with $ 2 \dt \vec{u} $ leads to, after applying integration by parts in the resultant, 
\begin{equation}\label{linest:u-H1}
	\begin{aligned}
		& \dfrac{d}{dt} \biggl( \mu \norm{\nabla \vec{u}}{\Lnorm{2}}^2 + (\mu +\lambda) \norm{\dv \vec{u}}{\Lnorm{2}}^2 \biggr) + 2 \norm{\rhod^{1/2}(Q_\mrm{m}^{+,o})^{1/2}\dt \vec u}{\Lnorm{2}}^2 \\
		& ~~~~ = - 2 \int \nabla p^o \cdot \dt \vec{u} \idx + 2 \int \rhod^{-1/2} \mathcal I_{\vec{u},1} \cdot \rhod^{1/2} \dt \vec{u} \idx \\ & ~~~~ + 2 \int \rhod^{-1/2}\mathcal I_{\vec{u},2} \cdot \rhod^{1/2} \dt \vec{u} \idx \\
 & ~~~~ {
		 \leq \norm{\rhod^{1/2} \dt \vec{u}}{\Lnorm{2}}  \norm{\rhod^{-1/2}\nabla p^o,\rhod^{-1/2}\mathcal I_{\vec{u},1},\rhod^{-1/2}\mathcal I_{\vec{u},2}}{\Lnorm{2}}}\\
& ~~~~{ \lesssim  \norm{\rhod^{1/2} \dt \vec{u}}{\Lnorm{2}}
		  \times \mathcal N(\norm{\rhod, \moistvar, \mathfrak T^o}{\Lnorm{\infty}} ) }\\
& ~~~~{\times \bigl( \norm{\nabla \rhod^{1/2}, \nabla \mathfrak q_\mrm{v}^o, \nabla \mathfrak T^o}{\Lnorm{2}}  + \norm{\vec u^o}{\Hnorm{1}} \norm{\nabla \vec u^o}{\Hnorm{1}} + \norm{\vec u^o}{\Hnorm{1}} + 1 \bigr),}
	\end{aligned}
\end{equation}
where we have substituted inequalities \eqref{ap-001} and \eqref{ap-002} in Appendix. 

On the other hand, after applying $ \dt $ to \eqref{eq:momentum-lin}, we arrive at
\begin{equation}\label{eq:dt-u-lin}
\begin{gathered}
\rhod Q_\mrm{m}^{+,o} \dt^2 \vec{u} - \mu \Delta\dt \vec{u} - (\mu+\lambda) \nabla \dv \dt \vec{u} \\
= - \dt (\rhod Q_\mrm{m}^{+,o}) \dt \vec{u} - \nabla\dt p^o + \dt \mathcal I_{\vec{u},1} + \dt \mathcal I_{\vec{u},2}.
\end{gathered}
\end{equation}
Taking the $ L^2 $-inner product of \eqref{eq:dt-u-lin} with $ 2 \dt \vec{u} $ leads to, after applying integration by parts, H\"older's inequality, and the interpolation inequality in the resultant, 
\begin{equation}\label{linest:dtu-L2}
	\begin{aligned}
	& \dfrac{d}{dt} \norm{\rhod^{1/2}(Q_\mrm{m}^{+,o})^{1/2}\dt \vec{u}}{\Lnorm{2}}^2 + 2\mu \norm{\nabla\dt \vec{u}}{\Lnorm{2}}^2 + 2(\mu+\lambda)\norm{\dv \dt \vec{u}}{\Lnorm{2}}^2 \\
	& = - \int \dt (\rhod Q_\mrm{m}^{+,o}) |\dt \vec{u}|^2 \idx + 2 \int \dt p^o \dv \dt \vec{u} \idx  + 2 \int \dt \mathcal{I}_{\vec{u},1} \cdot\dt u \idx \\
	& ~~~~ + 2 \int \dt \mathcal I_{\vec{u},2} \cdot \dt \vec{u} \idx \\
& {\leq 2 \norm{\dt p^o}{\Lnorm{2}} \norm{\dv \dt \vec{u}}{\Lnorm{2}} }\\
	 & ~~~~  + \norm{\rhod^{-1}\dt(\rhod Q_\mrm{m}^{+,o})}{\Lnorm{5}} \norm{\rhod^{1/2}\dt \vec{u}}{\Lnorm{2}}^{7/5} \norm{\rhod^{1/2}\dt \vec{u}}{\Lnorm{6}}^{3/5} \\
	& ~~~~ + 2 \norm{\rhod^{-1/2} \dt \mathcal I_{\vec{u},1}, \rhod^{-1/2} \dt \mathcal I_{\vec u, 2} }{\Lnorm{3/2}} \norm{ \rhod^{1/2} \dt \vec{u}}{\Lnorm{2}}^{1/2} \norm{\rhod^{1/2} \dt \vec{u}}{\Lnorm{6}}^{1/2}\\
	& \leq \mu \norm{\nabla \dt \vec u}{\Lnorm{2}}^2 +  \mathcal N (\norm{\rhod, \moistvar, \mathfrak T^o}{\Lnorm{\infty}}, \norm{\vec u^o}{\Hnorm{2}}, C_\varepsilon )\\
	& ~~~~ \times \biggl( \norm{\dt \rhod^{1/2}, \dt \mathfrak q_\mrm{v}^o, \dt \mathfrak T^o}{\Lnorm{2}}^2 
	 +(\norm{\dt \log \rhod, \lbrace \dt \mathfrak q_\mrm{j}^o \rbrace_{\mrm{j}\in \moistidx}}{\Lnorm{5}}^{10/7}+ 1)\\
	 & ~~~~ ~~~~ \times \norm{\rhod^{1/2}\dt \vec u}{\Lnorm{2}}^2 \biggr) + \varepsilon \norm{\dt \vec u^o }{\Hnorm{1}}^2 ,
	\end{aligned}
\end{equation}
{for any $\varepsilon \in (0,1)$,} where we have substituted inequalities \eqref{ap-001}, \eqref{ap-003}, and \eqref{ap-004} in Appendix, and applied Young's inequality, the Sobolev embedding inequalities, and the following inequalities (see, \cite[Lemma 3.2]{Feireisl2004}):
\begin{align*}
	& \norm{\rhod^{1/2}\dt \vec{j}}{\Lnorm{6}} \lesssim \norm{\rhod}{\Lnorm{\infty}}^{1/2} \norm{\dt \vec{j}}{\Hnorm{1}} \\
	& ~~~~ \lesssim \norm{\rhod}{\Lnorm{\infty}}^{1/2} (\norm{\rhod^{1/2}\dt \vec{j}}{\Lnorm{2}} + \norm{\nabla \dt \vec{j}}{\Lnorm{2}}), \\
	& \norm{\dt \vec{j}}{\Lnorm{6}} \lesssim \norm{\rhod^{1/2}\dt \vec{j}}{\Lnorm{2}} + \norm{\nabla \dt \vec{j}}{\Lnorm{2}}, ~~~~ \vec j \in \lbrace \vec u, \vec{u}^o \rbrace.
\end{align*}

Noticing that $ Q_\mrm{m}^{+,o} \geq 1 $, from \eqref{linest:u-H1} and \eqref{linest:dtu-L2}, after applying Gr\"onwall's inequality, one can conclude that,
\begin{equation}\label{linest:u-parabolic}
	\begin{aligned}
	& \sup_{0\leq s \leq t} \norm{\nabla \vec{u}(s), \rhod^{1/2} \dt \vec{u}(s)}{\Lnorm{2}}^2 + \int_0^t \norm{\rhod^{1/2} \dt \vec{u}(s), \nabla\dt \vec{u}(s) }{\Lnorm{2}}^2 \,ds \\
	& ~~~~ \leq c_{u,1} e^{c_{u,1}\int_0^t H_{u,1}(s) \,ds } \biggl( \int_0^t G_{u,1}(s) \,ds + \initial \biggr),
	\end{aligned}
\end{equation}
for some positive constant $ c_{u,1} \in (0,\infty) $, where
\begin{align}
	& \begin{aligned}
	& H_{u,1} := \mathcal N( \norm{\rhod, \moistvar, \mathfrak T^o}{\Lnorm{\infty}}, \norm{\vec u^o}{\Hnorm{2}}, C_\varepsilon )  \\
	& ~~~~ \times \bigl( 1 + \norm{\dt \log \rhod, \lbrace \dt \mathfrak q_\mrm{j}^o \rbrace_{\mrm{j}\in \moistidx}}{\Lnorm{5}}^{10/7} \bigr) ~~~~ \text{and}
	\end{aligned} \label{linsource:001}  \\
	& \begin{aligned}
	& G_{u,1} := \varepsilon \norm{\dt \vec u^o}{\Hnorm{1}}^2 + \mathcal N( \norm{\rhod, \moistvar, \mathfrak T^o}{\Lnorm{\infty}}, \norm{\vec u^o}{\Hnorm{2}}, C_\varepsilon ) \\
	& ~~~~  \times \bigl( \norm{ \nabla\rhod^{1/2}, \nabla \mathfrak q_\mrm{v}^o,\nabla \mathfrak T^o,\dt \rhod^{1/2}, \dt \mathfrak q_\mrm{v}^o , \dt \mathfrak T^o}{\Lnorm{2}}^2  + \norm{\vec u^o}{\Hnorm{2}}^2 \bigr). 
	\end{aligned} \label{linsource:002}
\end{align}

On the other hand, the following estimate follows by applying integration by parts and H\"older's inequality:
\begin{gather*}
	 \mu \norm{\nablah^2\vec{u}}{\Lnorm{2}}^2 + \mu \norm{\nablah \dz \vec{u}}{\Lnorm{2}}^2 + (\mu + \lambda) \norm{\nablah \dv \vec{u}}{\Lnorm{2}}^2 \\
	 ~~ = \int (\mu \Delta \vec{u} + (\mu + \lambda)\nabla\dv\vec{u}) \cdot \Deltah \vec{u} \idx \\
	 ~~ \leq \norm{\mu \Delta \vec{u} + (\mu + \lambda)\nabla\dv\vec{u}}{\Lnorm{2}} \norm{\nablah^2 \vec{u}}{\Lnorm{2}}.
\end{gather*}
In addition, notice that,
\begin{gather*}
	\biggl(\begin{array}{c}
	\mu \partial_{zz} \vec{v} \\ (2\mu + \lambda) \partial_{zz} w
	\end{array}\biggr) = \mu \Delta \vec{u} + (\mu+\lambda) \nabla\dv \vec{u} \\
	 - \biggl(\begin{array}{c}
	\mu \Deltah \vec{v} + (\mu + \lambda) \nablah \dv \vec{u} \\ \mu \Deltah w + (\mu+\lambda) \dz \dv_h \vec v
	\end{array}\biggr).
\end{gather*}
Consequently, using the above identities and \eqref{eq:momentum-lin}, one has that
\begin{equation}\label{linest:u-H2}
	\begin{aligned}
	\norm{\nabla^2 \vec{u}}{\Lnorm{2}} \lesssim & \norm{\mu \Delta \vec{u} + (\mu + \lambda)\nabla\dv\vec{u}}{\Lnorm{2}} \\
	\lesssim & \norm{\rhod}{\Lnorm{\infty}}^{1/2}\norm{Q_\mrm{m}^{+,o}}{\Lnorm{\infty}} \norm{\rhod^{1/2}\dt \vec{u}}{\Lnorm{2}} \\
	& + \norm{\nabla p^o}{\Lnorm{2}} + \norm{\mathcal I_{\vec{u},1}}{\Lnorm{2}} + \norm{\mathcal I_{\vec{u},2}}{\Lnorm{2}}.
	\end{aligned}
\end{equation}
Similarly, one can derive
\begin{align*}
	\norm{\nabla^2 \nablah \vec{u}}{\Lnorm{2}} \lesssim & \norm{\nablah (\mu \Delta \vec{u} + (\mu + \lambda)\nabla\dv\vec{u})}{\Lnorm{2}} ~~~~ \text{and} \\
	\norm{\partial_{zzz}\vec{u}}{\Lnorm{2}} \lesssim & \norm{\dz (\mu \Delta \vec{u} + (\mu + \lambda)\nabla\dv\vec{u})}{\Lnorm{2}} + \norm{\nabla^2 \nablah \vec{u}}{\Lnorm{2}}. 
\end{align*}
Hence, it follows that, after employing \eqref{eq:momentum-lin} again and applying H\"older's inequality, 
\begin{equation}\label{linest:u-H3}
	\begin{aligned}
	& \norm{\nabla^3 \vec{u}}{\Lnorm{2}} \lesssim  \norm{\nabla (\mu \Delta \vec{u} + (\mu + \lambda)\nabla\dv\vec{u})}{\Lnorm{2}} \\
	& ~~~~ \lesssim \norm{\rhod Q_\mrm{m}^{+,o}}{\Lnorm{\infty}} \norm{\dt \nabla \vec{u}}{\Lnorm{2}} + \norm{\rhod^{-1/2}\nabla(\rhod Q_\mrm{m}^{+,o})}{\Lnorm{6}} \norm{\rhod^{1/2}\dt \vec{u}}{\Lnorm{3}}  \\
	& ~~~~ ~~~~ + \norm{\nabla^2 p^o}{\Lnorm{2}} + \norm{\nabla \mathcal I_{\vec{u},1}}{\Lnorm{2}} + \norm{\nabla \mathcal I_{\vec{u},2}}{\Lnorm{2}}.
	\end{aligned}
\end{equation}

Therefore, we conclude, from \eqref{linest:u-H2} and  \eqref{linest:u-H3}, after substituting inequalities \eqref{ap-101} -- \eqref{ap-104}, that 
\begin{equation}\label{linest:u-elliptic-1}
	\sup_{0\leq s \leq t} \norm{\nabla^2 \vec{u}(s)}{\Lnorm{2}}^2 \leq c_{u,2} \sup_{0\leq s \leq t} H_{u,2}(s),
\end{equation}
and
\begin{equation}\label{linest:u-elliptic-2}
	\int_0^t \norm{\nabla^3 \vec{u} (s)}{\Lnorm{2}}^2 \,ds \leq c_{u,2} \int_0^t G_{u,2} (s) \,ds,
\end{equation}
for some positive constant $ c_{u,2} \in (0,\infty) $, where 
\begin{align}
	& \begin{aligned}
	& H_{u,2} := \norm{\rhod}{\Lnorm{\infty}} \bigl( \norm{\moistvar}{\Lnorm{\infty}}^2 + 1 \bigr) \norm{\rhod^{1/2}\dt \vec u}{\Lnorm{2}}^2\\
	& ~~~~ + \mathcal N( \norm{\mathfrak q_\mrm{v}^o,\mathcal T^o, \vec u^o}{\Hnorm{1}},\norm{\rhod}{\Lnorm{\infty}}, \norm{\moistvar}{\Lnorm{\infty}},1 )\\
	& ~~~~ \times (\norm{\nabla \rhod}{\Lnorm{6}}^2 
	 +  \norm{\nabla \vec u^o, \nabla \mathfrak q_\mrm{v}^o, \nabla \mathfrak T^o}{\Lnorm{3}}^2   
	 ), 
	\end{aligned}\label{linsource:003}\\
	& \begin{aligned}
	& G_{u,2} :=  \bigl(\norm{\rhod}{\Lnorm{\infty}}^2 + 1\bigr) \bigl( \norm{\moistvar}{\Lnorm{\infty}}^2 + 1 \bigr)\norm{\nabla \dt \vec{u}}{\Lnorm{2}}^2 \\
	& + \mathcal N(\norm{\rhod^{1/2}, \moistvar, \vec u^o, {\mathfrak T^o}}{\Hnorm{2}}) \bigl( \norm{\rhod^{1/2} \dt \vec u}{\Lnorm{2}}^2 + 1 \bigr).
	\end{aligned}\label{linsource:004}
\end{align}
\subsubsection*{The $ \mathfrak T, \lbrace \mathfrak q_\mrm{j} \rbrace_{\mrm j \in \moistidx} $-estimates} The estimates for $ \mathfrak T, \mathfrak q_\mrm{j} $, $ \mrm{j} \in \lbrace \mrm{v}, \mrm{c}, \mrm{r} \rbrace $, are similar, if not simpler, to the ones for $ \vec{u} $, we leave the {details} to the reader and only record the results here:
	\begin{align*}
	& \dfrac{d}{dt} \biggl( \kappa \norm{\nabla \mathfrak T}{\Lnorm{2}}^2 +  \norm{ \lbrace \nabla \mathfrak q_\mrm{j} \rbrace_{\mrm{j} \in \moistidx}, (Q_\mrm{th}^{+,o})^{1/2}\dt \mathfrak T, \lbrace \dt \mathfrak q_\mrm{j}\rbrace_{\mrm{j}\in \moistidx} }{\Lnorm{2}}^2 \biggr) \\
	& ~~~~ + 2 \kappa \norm{\nabla\dt \mathfrak T}{\Lnorm{2}}^2 + 2 \norm{(Q_\mrm{th}^{+,o})^{1/2} \dt \mathfrak T, \lbrace \dt \mathfrak q_\mrm{j} \rbrace_{\mrm{j} \in \moistidx}, \lbrace \nabla \dt \mathfrak q \rbrace_{\mrm{j} \in \moistidx} }{\Lnorm{2}}^2 \\
	& \lesssim \norm{\dt \mathfrak T, \lbrace \dt \mathfrak q_\mrm{j}\rbrace_{\mrm j \in \moistidx}}{\Lnorm{2}} \norm{\mathcal I_{\T, 1}, \mathcal I_{\T, 2}, \lbrace \mathcal I_{\mrm{j},1},\mathcal I_{\mrm{j},2} \rbrace_{\mrm{j} \in \moistidx} }{\Lnorm{2}} 
	 + \norm{\dt Q_\mrm{th}^{+,o}}{\Lnorm{5}} \\
	 & ~~~~ ~~~~ \times \norm{\dt \mathfrak T}{\Lnorm{2}}^{7/5} \norm{\dt \mathfrak T}{\Lnorm{6}}^{3/5}  
	 + \norm{\dt \mathcal I_{\T, 1}, \dt \mathcal I_{\T, 2} }{\Lnorm{3/2}} \norm{\dt \mathfrak T}{\Lnorm{2}}^{1/2} \norm{\dt \mathfrak T}{\Lnorm{6}}^{1/2} \\
	& ~~~~ + \norm{\lbrace \dt \mathcal I_{\mrm j, 1}, \dt \mathcal I_{\mrm j, 2} \rbrace_{\mrm{j} \in \moistidx} }{\Lnorm{3/2}} \norm{\lbrace \dt \mathfrak q_\mrm{j} \rbrace_{\mrm{j} \in \moistidx}}{\Lnorm{2}}^{1/2} \norm{\lbrace \dt \mathfrak q_\mrm{j} \rbrace_{\mrm{j} \in \moistidx}}{\Lnorm{6}}^{1/2},
	\end{align*}
and
\begin{align*}
& \norm{\nabla^2 \mathfrak T}{\Lnorm{2}} \lesssim \norm{ Q_\mrm{th}^{+,o}}{\Lnorm{\infty}} \norm{\dt \mathfrak T}{\Lnorm{2}}  + \norm{\mathcal I_{\T, 1}, \mathcal I_{\T, 2}}{\Lnorm{2}}, \\
	& \norm{\lbrace \nabla^2 \mathfrak q_\mrm{j} \rbrace_{\mrm{j} \in \moistidx}}{\Lnorm{2}} \lesssim \norm{\lbrace \dt \mathfrak q_\mrm{j}, \mathcal I_{\mrm{j},1}, \mathcal I_{\mrm{j},2} \rbrace_{\mrm{j} \in \moistidx} }{\Lnorm{2}},\\
	& \norm{\nabla^3 \mathfrak T}{\Lnorm{2}} \lesssim \norm{ Q_\mrm{th}^{+,o}}{\Lnorm{\infty}} \norm{\nabla \dt  \mathfrak T}{\Lnorm{2}}  + \norm{\nabla Q_\mrm{th}^{+.o}}{\Lnorm{6}} \norm{\dt \mathfrak T}{\Lnorm{3}} \\
	& ~~~~ + \norm{\nabla \mathcal I_{\T, 1},\nabla \mathcal I_{\T, 2}}{\Lnorm{2}},\\
	& \norm{\lbrace \nabla^3 \mathfrak q_\mrm{j} \rbrace_{\mrm{j} \in \moistidx}}{\Lnorm{2}}\lesssim \norm{\lbrace \nabla \dt \mathfrak q_\mrm{j}, \nabla \mathcal I_{\mrm{j},1}, \nabla \mathcal I_{\mrm{j},2} \rbrace_{\mrm{j} \in \moistidx} }{\Lnorm{2}}.
\end{align*}

Then, after substituting inequalities \eqref{ap-201} -- \eqref{ap-306}, applying the Sobolev embedding inequality and Young's inequality, one can conclude that,

\begin{gather}
	\begin{gathered}
	\sup_{0\leq s\leq t} \norm{\nabla \mathfrak T (s), \lbrace \nabla \mathfrak q_\mrm{j}(s) \rbrace_{\mrm{j}\in \moistidx}, \dt \mathfrak T(s),  \lbrace \dt \mathfrak q_\mrm{j}(s) \rbrace_{\mrm{j}\in \moistidx}}{\Lnorm{2}}^2 \\
	 + {\int_0^t }\norm{\dt \mathfrak T (s), \lbrace \dt \mathfrak q_\mrm{j}(s) \rbrace_{\mrm{j}\in \moistidx}}{\Hnorm{1}}^2 
	  \,ds \\
	  \leq c_o e^{c_{o}\int_0^t H_{o,1}(s) \,ds } \biggl( \int_0^t G_{o,1}(s) \,ds + \initial \biggr),
	\end{gathered} \label{linest:Total-parabolic} \\
	\begin{gathered}
	 \sup_{0\leq s\leq t}\norm{\nabla^2 \mathfrak T(s)}{\Lnorm{2}}^2 \leq c_{o} \sup_{0\leq s \leq t} H_{T,2}(s),
	 \end{gathered} \label{linest:temperature-elliptic} \\
	 \begin{gathered}
	 \sup_{0\leq s \leq t} \norm{\nabla^2 \mathfrak q_\mrm{v}(s),\nabla^2 \mathfrak q_\mrm{c}(s), \nabla^2 \mathfrak q_\mrm{r}(s)  }{\Lnorm{2}}^2 \leq c_o \sup_{0\leq s \leq t} H_{q,2}(s), \end{gathered} \label{linest:moist-elliptic} \\
	 \begin{gathered}
	 \int_0^t \norm{\nabla^3 \mathfrak T(s),\nabla^3 \mathfrak q_\mrm{v}(s),\nabla^3 \mathfrak q_\mrm{c}(s), \nabla^3 \mathfrak q_\mrm{r}(s)}{\Lnorm{2}}^2 \,ds
	 \leq  c_{o} \int_0^t G_{o,2} (s) \,ds,
	\end{gathered} \label{linest:Total-elliptic-2} 
\end{gather}
{for any positive constant $ \varepsilon \in (0,1) $ and} for some positive constant $ c_o \in (0,\infty) $, where
\begin{align}
	& \begin{aligned}
	H_{o,1}:= & \mathcal N(\norm{\vec u^o, \mathfrak T^o, \moistvar}{\Hnorm{2}}, \norm{\dz \log \rhod}{\Lnorm{6}} , C_\varepsilon ) \\
	& + \norm{\lbrace \dt \mathfrak q_\mrm{j}^o\rbrace_{\mrm j \in \moistidx}}{\Lnorm{5}}^{10/7},	\end{aligned} \label{linsource:005} \\
	& \begin{aligned}
	G_{o,1}:= & \varepsilon \norm{\nabla \vec{u}^o, \nabla \mathfrak T^o, \lbrace \nabla \mathfrak q_\mrm{j}^o \rbrace_{\mrm j \in \moistidx}, \dz \log \rhod }{\Lnorm{3}}^2 \\
	& + \varepsilon \norm{\dt \vec u^o, \dt \mathfrak T^o,\lbrace \dt \mathfrak q_\mrm{j}^o\rbrace_{\mrm j \in \moistidx},\dt \log \rhod}{\Hnorm{1}}^2,
	\end{aligned} \label{linsource:006}  \\
	& \begin{aligned}
	H_{T,2}:= & ( \norm{\moistvar}{\Lnorm{\infty}}^2 + 1) \\
&{\times \bigl( \norm{\dt \mathfrak T}{\Lnorm{2}}^2+ 1+ \norm{\vec u^o, \mathfrak T^o}{\Hnorm{1}}^2\norm{\nabla \vec u^o, \nabla \mathfrak T^o}{\Lnorm{3}}^2 \bigr) }\\
	& + \mathcal N(\norm{\vec u^o, \mathfrak T^o, \moistvar}{\Hnorm{1}}, \norm{\dz \log\rhod }{\Lnorm{3}},1 ),
	\end{aligned} \label{linsource:007} \\
	& \begin{aligned}
	H_{q,2}:= &  \norm{\lbrace \dt \mathfrak q_\mrm{j}\rbrace_{\mrm j \in \moistidx}}{\Lnorm{2}}^2 + (1+\norm{\vec u^o}{\Hnorm{1}}^2 ) \norm{ \lbrace \nabla \mathfrak q_\mrm{j}^o\rbrace_{\mrm j \in \moistidx}}{\Lnorm{3}}^2  \\
	 & + \mathcal N(\norm{\vec u^o, \mathfrak T^o, \moistvar}{\Hnorm{1}}, \norm{\dz \log\rhod }{\Lnorm{3}},1 ) ,
	\end{aligned} \label{linsource:008} \\
	& \begin{aligned}
	G_{o,2} := &( \norm{\moistvar}{\Lnorm{\infty}}^2 + 1 ) \norm{ \nabla \dt  \mathfrak T, \lbrace \nabla \dt \mathfrak q_\mrm{j}\rbrace_{\mrm{j}\in \moistidx} }{\Lnorm{2}}^2 \\
	& + \norm{\lbrace \nabla \mathfrak q_\mrm{j}^o \rbrace_{\mrm j \in \moistidx}}{\Lnorm{6}}^2 \norm{\dt \mathfrak T}{\Lnorm{3}}^2 \\
	& + \mathcal N(\norm{\vec u^o, \mathfrak T^o, \moistvar, \log \rhod}{\Hnorm{2}},1).
	\end{aligned} \label{linsource:009}
\end{align}

\subsubsection*{Summarizing the estimates}

It follows from H\"older's inequality that, 
\begin{gather*}
	\int_0^t \norm{\nabla \vec u^o (s)}{\Hnorm{2}} \,ds \leq t^{1/2} \bigl(\int_0^t \norm{\nabla \vec u^o (s)}{\Hnorm{2}}^2 \,ds \bigr)^{1/2} \\
	 \leq (\mathfrak c_{u,1} + \mathfrak c_{u,2})t + \mathfrak c_{u,3} t^{1/2} \leq 1,
\end{gather*}
for $ t \in (0, T_1^* ] $ with $ T_1^* \in (0,1) $ small enough. 
Then \eqref{linest:d-1} and \eqref{linest:d-2} imply {that} there exist positive constants $ \mathfrak c_{d,1}, \mathfrak c_{d,2}\in (0,\infty) $, such that
\begin{equation}\label{unest:density-lin}
\begin{gathered}
	\sup_{0\leq s \leq t} \norm{\rhod^{1/2}(s), \log \rhod(s)}{\Hnorm{2}}^2 \leq \mathfrak c_{d,1} ( 1 + \mathfrak c_\mrm{in} ), \\
	\sup_{0\leq s \leq t}\norm{\dt \rhod^{1/2}(s), \dt \log \rhod(s)}{\Hnorm{1}}^2
 \leq \mathfrak c_{d,2}(\mathfrak c_{u, 1} + \mathfrak c_{u, 2})( 1 + \mathfrak c_\mrm{in}), 
 \end{gathered}
 \end{equation}
 for $ t \in (0, T_1^* ] $. 
 
 From the definitions of $ G_{u,1}, G_{o,1} $ in \eqref{linsource:002} and \eqref{linsource:006}, after applying the Sobolev embedding inequality, one can conclude that
 \begin{gather*}
 	\int_0^t \bigl( G_{u,1}(s) + G_{o,1}(s) \bigr) \,ds \leq \varepsilon (\mathfrak c_{u,1} + \mathfrak c_{o,1} ) \\
 	+ \mathcal N (\mathfrak c_{u,1}, \mathfrak c_{u,2}, \mathfrak c_{o,1}, \mathfrak c_{T} , \mathfrak c_{q}, C_\varepsilon ) t
 	\leq 1,
 \end{gather*}
 for $ \varepsilon \leq \varepsilon^* $ and $ t \in (0,T_2^*] $, with some $ \varepsilon^* \in (0,1) $ and $ T_2^* \in(0, T_1^*] $ small enough. On the other hand, from the definitions of $ H_{u,1}, H_{o,1} $ in \eqref{linsource:001} and \eqref{linsource:005}, after applying the Sobolev embedding inequality, one has
 \begin{gather*}
 	\int_0^t \bigl( H_{u,1}(s) + H_{o,1}(s) \bigr) \,ds )\leq \mathcal N (\mathfrak c_{u,1}, \mathfrak c_{u,2}, \mathfrak c_{o,1}, \mathfrak c_{T} , \mathfrak c_{q}, C_\varepsilon ) \\
 	\times \int_0^t \bigl(1 + \norm{\lbrace \dt \mathfrak q_{\mrm{j}}^o(s) \rbrace_{\mrm j \in \moistidx} }{\Lnorm{5}}^{10/7} \bigr) \,ds\\
 	\leq  \mathcal N (\mathfrak c_{u,1}, \mathfrak c_{u,2}, \mathfrak c_{o,1}, \mathfrak c_{T} , \mathfrak c_{q}, C_\varepsilon ) \bigl( t + t^{5/14} \bigr) \leq 1,
 \end{gather*}
 for $ t \in (0,T_3^*]$, with some $ T_3^* \in (0, T_2^*] $ small enough, 
 where we have used the fact that, as a consequence of the Gagliardo-Nirenberg interpolation inequality and H\"older's inequality, 
  \begin{align*}
 	& \int_0^t  \norm{\lbrace \dt \mathfrak q_{\mrm{j}}^o(s) \rbrace_{\mrm j \in \moistidx} }{\Lnorm{5}}^{10/7} \,ds \leq t^{5/14} \\
 	& ~~ \times \bigl( \int_0^t \norm{\lbrace \dt \mathfrak q_{\mrm{j}}^o(s) \rbrace_{\mrm j \in \moistidx} }{\Lnorm{2}}^{2/9}\norm{\lbrace \dt \mathfrak q_{\mrm{j}}^o(s) \rbrace_{\mrm j \in \moistidx} }{\Hnorm{1}}^{2}  \bigr)^{9/14}.
 \end{align*}
Therefore \eqref{linest:u-parabolic} and \eqref{linest:Total-parabolic} imply that, for $ t \in (0,T_3^*] $, 
\begin{equation}\label{unest:parabolic}
	\begin{aligned}
		& \sup_{0\leq s\leq t} \bigl( \norm{\rhod^{1/2}\dt \vec u(s), \dt  \mathfrak T(s), \lbrace \dt \mathfrak q_\mrm{j}(s) \rbrace_{\mrm j \in \moistidx} }{\Lnorm{2}}^2  \\
		& ~~~~ ~~~~ + \norm{\vec{u}(s), \mathfrak T(s), \lbrace \mathfrak q_\mrm{j}(s)\rbrace_{\mrm j \in \moistidx}}{\Hnorm{1}}^2 \bigr) \\
		& ~~~~ + \int_0^t \bigl( \norm{\dt \vec u(s) ,\dt  \mathfrak T(s), \lbrace \dt \mathfrak q_\mrm{j}(s) \rbrace_{\mrm j \in \moistidx} }{\Hnorm{1}}^2 \\
		& ~~~~ ~~~~ + \norm{\rhod^{1/2} \dt \vec{u}(s)}{\Lnorm{2}}^2 \bigr) \,ds \leq c_1 (1 + \mathfrak c_\mrm{in}),
	\end{aligned}
\end{equation}
for some positive constant $ c_1 \in (0,\infty) $, where we have used the fundamental theorem of calculus and the embedding inequality as below:
\begin{gather*}
	\norm{\vec u(s), \mathfrak T(s), \lbrace \mathfrak q_\mrm{j}(s)\rbrace_{\mrm j \in \moistidx}}{\Lnorm{2}}\leq \norm{\vec u(0),\mathfrak T(0), \lbrace \mathfrak q_\mrm{j}(0)\rbrace_{\mrm j \in \moistidx}}{\Lnorm{2}} \\
	 + \int_0^s \norm{\dt \vec{u}(\tau),\dt \mathfrak T(\tau), \lbrace \dt \mathfrak q_\mrm{j}(\tau)\rbrace_{\mrm j \in \moistidx}}{\Lnorm{2}} \,d\tau, \\
	\norm{\dt \vec u}{\Lnorm{2}} \lesssim \norm{\rhod^{1/2} \dt \vec u}{\Lnorm{2}} + \norm{\nabla \dt \vec u}{\Lnorm{2}}.
\end{gather*}

Next, applying the Gagliardo-Nirenberg inequality yields
\begin{gather*}
	\norm{\lbrace \nabla \mathfrak q_{\mrm{j}}^o \rbrace_{\mrm{j}\in \moistidx}}{\Lnorm{3}}^2 \lesssim \norm{\lbrace \nabla \mathfrak q_{\mrm{j}}^o \rbrace_{\mrm{j}\in \moistidx}}{\Lnorm{2}}^2 \\
	 + \norm{\lbrace \nabla \mathfrak q_{\mrm{j}}^o \rbrace_{\mrm{j}\in \moistidx}}{\Lnorm{2}} \norm{\lbrace \nabla^2 \mathfrak q_{\mrm{j}}^o \rbrace_{\mrm{j}\in \moistidx}}{\Lnorm{2}}\\
	 \leq \mathfrak c_{o,1} + \mathfrak c_{o,1}^{1/2} \mathfrak c_{q}^{1/2}.
\end{gather*}
Then applying Young's inequality, \eqref{unest:parabolic}, and\eqref{unest:density-lin} to $ H_{q,2} $ in \eqref{linsource:008}, from \eqref{linest:moist-elliptic}, yields
\begin{equation}\label{unest:elliptic-q}
	\sup_{0\leq s\leq t} \norm{\lbrace\nabla^2 \mathfrak q_\mrm{j}(s)\rbrace_{\mrm{j}\in \moistidx} }{\Lnorm{2}}^2 \leq \dfrac{1}{2} \mathfrak c_q + \mathcal N_q(\mathfrak c_{o,1}, \mathfrak c_\mrm{in} ),
\end{equation}
for $ t \in (0,T_3^*] $ and some smooth function $ \mathcal N_q(\cdot) $.
Similarly, from \eqref{linest:u-elliptic-1} and \eqref{linest:temperature-elliptic}, one can conclude that
\begin{equation}\label{unest:elliptic-t-u}
	\sup_{0\leq s\leq t} \norm{\nabla^2 \vec{u}(s), \nabla^2 \mathfrak T}{\Lnorm{2}}^2 \leq \dfrac{1}{4} (\mathfrak c_{u,2} + \mathfrak c_T) + \mathcal N_u (\mathfrak c_{u,1}, \mathfrak c_{o,1}, \mathfrak c_q, \mathfrak c_\mrm{in} ),
\end{equation}
for $ t \in (0,T_3^*] $ and some smooth function $ \mathcal N_u (\cdot) $.
Moreover, it is easy to check that 
\begin{gather*}
G_{u,2} + G_{o,2} \leq \mathcal N_{o}'(\norm{\rhod^{1/2}, \log\rhod, \vec u^o, \mathfrak T^o , \moistvar }{\Hnorm{2}}) \\
 \times ( 1 + \norm{\dt \vec u,\dt \mathfrak T, \lbrace \dt \mathfrak q_\mrm{j}\rbrace_{\mrm{j}\in \moistidx} }{\Hnorm{1}}^2),
\end{gather*}
for some smooth function $ \mathcal N_o' (\cdot) $
Therefore, from \eqref{linest:u-elliptic-2} and \eqref{linest:Total-elliptic-2}, after applying \eqref{unest:parabolic}, one can conclude that
\begin{equation}\label{unest:elliptic-o-2}
	\begin{gathered}
	\int_0^t \norm{\nabla^3 \vec u(s), \nabla^3 \mathfrak T(s), \lbrace \nabla^3 \mathfrak q_\mrm{j}(s)\rbrace_{\mrm j \in \moistidx}}{\Lnorm{2}}^2 \,ds \\
	 \leq \mathcal N_{o} ( \mathfrak c_{u,1},\mathfrak c_{u,2}, \mathfrak c_{o,1}, \mathfrak c_{T}, \mathfrak c_{q} ,\mathfrak c_{in}),
	\end{gathered}
\end{equation}
for $ t \in (0,T_3^*] $ and some smooth function $ \mathcal N_o (\cdot) $.

Now, let
\begin{equation}\label{id:constants}
	\begin{gathered}
		\mathfrak c_{u,1} = \mathfrak c_{o, 1} = c_1 ( 1 + \mathfrak c_\mrm{in}) , \\
		\mathfrak c_{q} = 4 \mathcal N_q(\mathfrak c_{0,1},\mathfrak c_\mrm{in}), \\
		\mathfrak c_{u,2} = \mathfrak c_{T} = 4 \mathcal N_u(\mathfrak c_{u,1},\mathfrak c_{o,1}, \mathfrak c_{q},\mathfrak c_\mrm{in} ), \\
		\mathfrak c_{u,3} = \mathfrak c_{o,2} = \mathcal N_{o}( \mathfrak c_{u,1},\mathfrak c_{u,2}, \mathfrak c_{o,1}, \mathfrak c_{T}, \mathfrak c_{q} ,\mathfrak c_{in}),
	\end{gathered}
\end{equation}
and $ T^* = T^*_3 $,
where 
$ c_1, \mathcal N_q(\cdot ) ,  \mathcal N_u(\cdot) , \mathcal N_{o}(\cdot) $
are defined in 
\eqref{unest:parabolic},  \eqref{unest:elliptic-q}, \eqref{unest:elliptic-t-u}, and \eqref{unest:elliptic-o-2}, respectively. 

Then $ (\vec u, \mathfrak T, \lbrace \mathfrak q_\mrm j \rbrace_{\mrm j \in \moistidx} ) \in \mathfrak Y $, and we have constructed a map 
\begin{equation}\label{def:mapping}
\mathcal M : \mathfrak Y \mapsto \mathfrak Y  ~~ \text{by} ~~ \mathcal M(\vec u^o, \mathfrak T^o, \moistvar) := (\vec u, \mathfrak T,\lbrace \mathfrak q_\mrm j \rbrace_{\mrm j \in \moistidx} ), 
\end{equation} 
where $ (\vec u, \mathfrak T,\lbrace \mathfrak q_\mrm j \rbrace_{\mrm j \in \moistidx} ) $ is the solution to \eqref{eq:momentum-lin} -- \eqref{eq:rain-ratio-lin}, constructed above. 

Moreover, \eqref{unest:density-lin} implies that $ 0 < {\rhod'} < \rhod < \rhod'' < \infty $, for some positive and finite constants $ {\rhod'} $ and $ \rhod'' $, depending only on $ \mathfrak c_\mrm{in} $. \eqref{unest:density-lin} -- \eqref{unest:elliptic-o-2} yield that, 
there exists a constant $ \mathfrak c_o $, depending on $ \mathfrak c_{in} $, such that the following estimate holds: 
\begin{equation}\label{unest:total}
	\begin{aligned}
	& \sup_{0\leq s \leq T^*} \bigl(\norm{\rhod^{1/2}(s), \log \rhod (s), \vec u (s), \mathfrak T (s), \lbrace \mathfrak q_\mrm j (s)\rbrace_{\mrm j \in \moistidx} }{\Hnorm{2}}^2 \\
	& ~~~~ + \norm{\dt \rhod^{1/2}(s), \dt \log \rhod(s)}{\Hnorm{1}}^2 \\
	& ~~~~ + 
	\norm{\dt \vec u(s), \dt \mathfrak T(s), \lbrace \dt \mathfrak q_\mrm j (s) \rbrace_{\mrm j \in \moistidx}}{\Lnorm{2}}^2 \bigr) \\
	& + \int_0^{T^*} \norm{\vec u (s), \mathfrak T (s), \lbrace \mathfrak q_\mrm j (s)\rbrace_{\mrm j \in \moistidx} }{\Hnorm{3}}^2 \\
	& ~~~~ + 
	\norm{\dt \vec u(s), \dt \mathfrak T(s), \lbrace \dt \mathfrak q_\mrm j (s) \rbrace_{\mrm j \in \moistidx}}{\Hnorm{1}}^2 \,ds \leq \mathfrak c_o. 
	\end{aligned}
\end{equation}

\subsection{The {Contraction} mapping}\label{sec:contracting}

Now we consider $ \lbrace (\vec u^o_i, \mathfrak T^o_i, \lbrace \mathfrak q_{\mrm j, i}^o \rbrace_{\mrm j \in \moistidx}) \in \mathfrak Y  \rbrace_{i=0,1} $, and let $ (\vec u_i, \mathfrak T_i, \lbrace \mathfrak q_{\mrm j, i} \rbrace_{\mrm j \in \moistidx}) = \mathcal M (\vec u^o_i, \mathfrak T^o_i, \lbrace \mathfrak q_{\mrm j, i}^o \rbrace_{\mrm j \in \moistidx}) $, $ i = 1,2 $, be the solution map defined in \eqref{def:mapping}. The solutions to \eqref{eq:dry-air-mass-lin} with $ \vec u^o = \vec u^o_i $, $ i = 1,2 $, are denoted as $ \rho_{\mrm d, i} $. 
Let $ \delta (\cdot) $ denote the action of taking difference of variable with indices $ i = 1 $ and $ i = 2 $, i.e., $ \delta f = f_1 - f_2 $. 
For instance, 
\begin{gather*}
	\delta \rhod = \rho_{\mrm d , 1} - \rho_{\mrm d,2}, \\
	\delta \vec u = \vec u_1 - \vec u_2, ~ \delta \mathfrak T = \mathfrak T_1 - \mathfrak T_2, ~ 
	\lbrace \delta \mathfrak q_\mrm j = \mathfrak q_{\mrm j , 1} - \mathfrak q_{\mrm j, 2} \rbrace_{\mrm j \in \moistidx} ,\\
	\delta \vec u^o = \vec u_1^o - \vec u_2^o, ~ \delta \mathfrak T^o = \mathfrak T_1^o - \mathfrak T_2^o, ~
	\lbrace \delta \mathfrak q_\mrm j^o = \mathfrak q_{\mrm j , 1}^o - \mathfrak q_{\mrm j, 2}^o \rbrace_{\mrm j \in \moistidx}.
\end{gather*} 
Then $ \delta \rhod $ and $ (\delta \vec u, \delta\mathfrak T, \lbrace \delta \mathfrak q_\mrm j \rbrace_{\mrm j \in \moistidx} ) $ satisfy the following equations:
\begin{gather}
	\dt \delta \rhod + \dv (\rho_{\mrm d, 1} \delta \vec u^o) + \dv (\delta \rhod \vec u_2^o) = 0, \label{eq:dry-air-mass-lin-ct} \\
	\begin{gathered}
	\rho_{\mrm d, 1} Q_{\mrm m,1}^{+,o} \dt \delta \vec u - \mu \Delta \delta \vec u - (\mu +\lambda) \nabla \dv \delta \vec u = - \delta (\rhod Q_\mrm m ^{+,o}) \dt \vec u_2 - \nabla \delta p^o \\
	 + \delta \mathcal{I}_{\vec u, 1} + \delta \mathcal I_{\vec u, 2}, 
	 \end{gathered} \label{eq:momentum-lin-ct}\\
	 Q_{\mrm{th}, 1}^{+,o} \dt \delta \mathfrak T - \kappa \Delta \delta \mathfrak T = - \delta Q_\mrm{th}^{+,o} \dt \mathfrak T_2 + \delta \mathcal I_{\T, 1} + \delta \mathcal I_{\T, 2}, \label{eq:temperature-lin-ct} \\
	 \dt \delta \mathfrak q_\mrm j - \Delta \delta \mathfrak q_\mrm j = \delta \mathcal I_{\mrm j, 1} + \delta \mathcal I_{\mrm j, 2}, ~~~~ \mrm j \in \lbrace \mrm v, \mrm c \rbrace, \label{eq:moist-ratio-lin-ct} \\
	 \begin{gathered}
	 \dt \delta \mathfrak q_\mrm r - \Delta \delta \mathfrak q_\mrm r = \delta \mathcal I_{\mrm r, 1} + \delta [ \mathcal I_{\mrm r, 2} - (\mathfrak q_\mrm r^o + \psi_\mrm r)\Vr \dz \log \rhod ] \\
	 + \delta [(\mathfrak q_\mrm r^o + \psi_\mrm r)\Vr \dz \log \rhod].\end{gathered} \label{eq:moist-ratio-lin-ct-2}
\end{gather}
Then, applying the standard $L^2 $-estimates to \eqref{eq:dry-air-mass-lin-ct} 
-- \eqref{eq:moist-ratio-lin-ct} yields,
\begin{equation}\label{ctest:density}
	\begin{gathered}
		\dfrac{d}{dt} \norm{\delta \rhod}{\Lnorm{2}}^2 \lesssim \norm{\nabla \vec u_2^o}{\Hnorm{2}} \norm{\delta\rhod}{\Lnorm{2}}^2 \\
		 + \norm{\rho_{\mrm d, 1}}{\Hnorm{2}} \norm{\delta \rhod}{\Lnorm{2}} \norm{\delta \vec u^o}{\Hnorm{1}},
	\end{gathered}
\end{equation}
and
\begin{equation}\label{ctest:total}
	\begin{gathered}
	\dfrac{d}{dt} \norm{\rho_{\mrm d, 1}^{1/2}(Q_{\mrm{m},1}^{+,o})^{1/2} \delta \vec u, (Q_{\mrm{th}, 1}^{+,o})^{1/2} \delta \mathfrak T, \lbrace \delta \mathfrak q_\mrm j \rbrace_{\mrm j \in \moistidx}}{\Lnorm{2}}^2  
	  + 2 \mu \norm{\nabla \delta \vec u}{\Lnorm{2}}^2 \\
	 + 2 (\mu +\lambda ) \norm{\dv \delta \vec u}{\Lnorm{2}}^2 + 2 \kappa \norm{\nabla \delta \mathfrak T}{\Lnorm{2}}^2 + 2\norm{\lbrace \nabla \delta  \mathfrak q_\mrm j \rbrace_{\mrm j \in \moistidx}}{\Lnorm{2}}^2 \\
	 \leq  \norm{\delta (\rhod Q_\mrm{m}^{+,o}), \delta Q_\mrm{th}^{+,o}}{\Lnorm{2}} \norm{\dt \vec u_2,\dt \mathfrak T_2}{\Lnorm{3}} \norm{\delta \vec u, \delta \mathfrak T}{\Lnorm{6}}\\
	 + \norm{\dt(Q_{\mrm{th},1}^{+,o})}{\Lnorm 
	 2} \norm{\delta \mathfrak T}{\Lnorm{3}} \norm{\delta \mathfrak T}{\Lnorm{6}}\\ + \norm{\dt(\rho_{\mrm d, 1}Q_{\mrm{m},1}^{+,o})}{\Lnorm 
	 2} \norm{\delta \vec u}{\Lnorm{3}} \norm{\delta \vec u}{\Lnorm{6}}
	 + \norm{\delta p^o}{\Lnorm{2}} \norm{\nabla \delta \vec u}{\Lnorm{2}}\\
	  + \norm{\lbrace \delta \mathcal I_{\mrm j, 1}, \delta \mathcal I_{\mrm j, 2}\rbrace_{\mrm j \in \lbrace \vec u, \T,\mrm v, \mrm{c}} \rbrace }{\Lnorm{3/2}} \norm{\delta \vec u, \delta \mathfrak T, \lbrace \delta \mathfrak q_\mrm j \rbrace_{\mrm j \in \moistidx}}{\Lnorm{3}}\\
	  + \norm{\delta [\mathcal I_{\mrm r, 2} - (\mathfrak q_{\mrm r}^o + \psi_\mrm r)\Vr \dz \log \rhod]}{\Lnorm{3/2}} \norm{\delta \mathfrak q_\mrm r}{\Lnorm{3}}\\
	  +  \norm{\delta[(\mathfrak q_{\mrm r}^o+\psi_{\mrm r}^o)\Vr \log \rhod]}{\Lnorm{2}}\norm{\dz \delta \mathfrak q_\mrm r}{\Lnorm{2}} \\
	   + \norm{\delta [\dz\bigl((\mathfrak q_{\mrm r}^o + \psi_{\mrm r}^o)\Vr\bigr) \log \rhod] }{\Lnorm{3/2}} \norm{\delta \mathfrak q_\mrm r}{\Lnorm{3}}
	  .
	\end{gathered}
\end{equation}
One can easily check, after applying Young's inequality, Gr\"onwall's inequality, and H\"older's inequality to \eqref{ctest:density}, that
\begin{equation}\label{ctest:001}
	\sup_{0\leq s \leq t} \norm{\delta \rhod(s)}{\Lnorm{2}}^2 \leq \varepsilon' c_1 e^{c_1 ( 1 + C_{\varepsilon'} t^{1/2})t^{1/2} }\times \int_0^t \norm{\delta \vec u^o(s)}{\Hnorm{1}}^2 \,ds,
\end{equation}
for $ t \in (0,T^*] $ and any $ \varepsilon' \in (0,1) $, where we have used the uniform estimate in \eqref{unest:total}.

The estimates of the $ L^{3/2} $-norms
in \eqref{ctest:total} are tedious but straightforward. We skip the details and only record the result here:
\begin{align*}
	& \norm{\lbrace \delta \mathcal I_{\mrm j, 1}, \delta \mathcal I_{\mrm j, 2}\rbrace_{\mrm j \in \lbrace \vec u, \T,\mrm v, \mrm{c}} \rbrace }{\Lnorm{3/2}} + \norm{\delta [\mathcal I_{\mrm r, 2}-(\mathfrak q_{\mrm r}^o + \psi_\mrm r)\Vr \dz \log \rhod]}{\Lnorm{3/2}} \\
	& ~~~~ ~~~~  + \norm{\delta [\dz\bigl((\mathfrak q_{\mrm r}^o + \psi_{\mrm r}^o)\Vr\bigr) \log \rhod] }{\Lnorm{3/2}} \\
	& ~~~~
	 \lesssim \mathcal N( \norm{\lbrace \rho_{\mrm d, i}, \vec u^o_i, \mathfrak T_i^o, \lbrace \mathfrak q_{\mrm j, i}^o \rbrace_{\mrm j \in \moistidx} \rbrace_{i=1,2}}{\Hnorm{2}},1) \\
	& ~~~~ ~~~~ \times \bigl( \norm{\delta \rhod}{\Lnorm{2}} + \norm{\delta \vec u^o, \delta \mathfrak T^o, \lbrace \delta \mathfrak q_\mrm j^o \rbrace_{\mrm j \in \moistidx}}{\Hnorm{1}} \bigr).
\end{align*}
Therefore, after applying the Gagliardo-Nirenberg inequality, it follows from \eqref{ctest:total} that
\begin{equation}\label{ctest:total-2}
\begin{gathered}
	\dfrac{d}{dt} \norm{\rho_{\mrm d, 1}^{1/2}(Q_{\mrm{m},1}^{+,o})^{1/2} \delta \vec u, (Q_{\mrm{th}, 1}^{+,o})^{1/2} \delta \mathfrak T, \lbrace \delta \mathfrak q_\mrm j \rbrace_{\mrm j \in \moistidx}}{\Lnorm{2}}^2  
	+ 2 \mu \norm{\nabla \delta \vec u}{\Lnorm{2}}^2 \\
	+ 2 (\mu +\lambda ) \norm{\dv \delta \vec u}{\Lnorm{2}}^2 + 2 \kappa \norm{\nabla \delta \mathfrak T}{\Lnorm{2}}^2 + 2\norm{\lbrace \nabla \delta  \mathfrak q_\mrm j \rbrace_{\mrm j \in \moistidx}}{\Lnorm{2}}^2 \\
	\leq \mathcal N( \norm{\lbrace \rho_{\mrm d, i}, \vec u^o_i, \mathfrak T_i^o, \lbrace \mathfrak q_{\mrm j, i}^o \rbrace_{\mrm j \in \moistidx} \rbrace_{i=1,2}}{\Hnorm{2}},1) \\
	\times \bigl( \norm{\delta \rhod}{\Lnorm{2}} + \norm{\lbrace \delta \mathfrak q_\mrm j ^o \rbrace_{\mrm j \in \moistidx}}{\Lnorm{2}} \bigr) \\
		\times \bigl( 1 + \norm{\dt \vec u_2, \dt \mathfrak T_2}{\Lnorm{2}}^{1/2}\norm{\dt \vec u_2, \dt \mathfrak T_2}{\Hnorm{1}}^{1/2} \bigr) \norm{\delta \vec u, \delta \mathfrak T}{\Hnorm{1}} \\
	+ \mathcal N ( \norm{\lbrace \rho_{\mrm d, i}, \lbrace \mathfrak q_{\mrm j, i}^o \rbrace_{\mrm j \in \moistidx} \rbrace_{i=1,2}}{\Hnorm{2}},1) \norm{\dt \rho_{\mrm d,1}, \lbrace \dt \mathfrak q_{\mrm j}^o\rbrace_{\mrm j \in \moistidx}}{\Lnorm{2}}\\
	\times \norm{\delta \mathfrak T, \delta \vec u}{\Lnorm{2}}^{1/2} \norm{\delta \mathfrak T, \delta \vec u}{\Hnorm{1}}^{3/2} \\
	+ \mathcal N( \norm{\lbrace \rho_{\mrm d, i}, \vec u^o_i, \mathfrak T_i^o, \lbrace \mathfrak q_{\mrm j, i}^o \rbrace_{\mrm j \in \moistidx} \rbrace_{i=1,2}}{\Hnorm{2}},1) \bigl( \norm{\delta \rhod}{\Lnorm{2}} \\
	 + \norm{\delta \vec u^o, \delta \mathfrak T^o, \lbrace \delta \mathfrak q_\mrm j^o \rbrace_{\mrm j \in \moistidx}}{\Hnorm{1}} \bigr) \\
		\times \norm{\delta \vec u, \delta \mathfrak T, \lbrace \delta \mathfrak q_\mrm j \rbrace_{\mrm j \in \moistidx}}{\Lnorm{2}}^{1/2}\norm{\delta \vec u, \delta \mathfrak T, \lbrace \delta \mathfrak q_\mrm j \rbrace_{\mrm j \in \moistidx}}{\Hnorm{1}}^{1/2}\\
	+ \mathcal N(\norm{\lbrace\rho_{\mrm d, i}, \mathfrak q_{\mrm r, i}^o\rbrace_{i = 1,2 }}{\Hnorm{2}},1) \norm{\delta \mathfrak q_\mrm r}{\Hnorm{1}} \norm{\delta \mathfrak q_{\mrm r}^o, \delta \rhod}{\Lnorm{2}}.
\end{gathered}
\end{equation}
Thus, again, after applying Young's inequality and H\"older's inequality to the above inequality, one can derive that
\begin{equation}\label{ctest:002}
	\begin{aligned}
		& \sup_{0\leq s \leq t} \norm{\delta \vec u, \delta \mathfrak T, \lbrace \delta \mathfrak q_{\mrm j} \rbrace_{\mrm j \in \moistidx}}{\Lnorm{2}}^2 + \int_0^t \norm{\nabla\delta \vec u, \nabla \delta \mathfrak T, \lbrace \nabla \delta \mathfrak q_\mrm j \rbrace_{\mrm j \in \moistidx}}{\Lnorm{2}}^2 \,ds \\
		& \leq \varepsilon'' \int_0^t \norm{\delta \vec u, \delta \mathfrak T, \lbrace \delta \mathfrak q_{\mrm j} \rbrace_{\mrm j \in \moistidx}, \delta \vec u^o, \delta \mathfrak T^o, \lbrace \delta \mathfrak q_{\mrm j}^o \rbrace_{\mrm j \in \moistidx}}{\Hnorm{1}}^2 \,ds \\
		& ~~~~ + c_2 C_{\varepsilon''} \int_0^t ( 1 + \norm{\dt \vec u_2, \dt \mathfrak T_2}{\Hnorm{1}})   (\norm{\delta \rhod}{\Lnorm{2}}^2 + 1)  \,ds \\
		& ~~~~ + c_2 C_{\varepsilon''} \int_0^t \norm{\delta \vec u, \delta \mathfrak T, \lbrace \delta \mathfrak q_\mrm j \rbrace_{\mrm j\in \moistidx}}{\Lnorm{2}}^2 \,ds \\
		& \leq (\varepsilon'' + c_2 C_{\varepsilon''} ) t \times \sup_{0\leq s\leq t} \norm{\delta \vec u, \delta \mathfrak T, \lbrace \delta \mathfrak q_{\mrm j} \rbrace_{\mrm j \in \moistidx}}{\Lnorm{2}}^2 \\
		& ~~~~
		+ \varepsilon'' \int_0^t \norm{\nabla\delta \vec u, \nabla \delta \mathfrak T, \lbrace \nabla \delta \mathfrak q_\mrm j \rbrace_{\mrm j \in \moistidx}}{\Lnorm{2}}^2 \,ds \\
		& ~~~~ + \varepsilon'' \int_{0}^t \norm{\delta \vec u^o, \delta \mathfrak T^o, \lbrace \delta \mathfrak q_{\mrm j}^o \rbrace_{\mrm j \in \moistidx}}{\Hnorm{1}}^2 \,ds \\
		& ~~~~ + c_2 C_{\varepsilon''} (t + t^{1/2}\mathfrak c_o^{1/2}) \bigl(1 + \sup_{0\leq s\leq t} \norm{\delta\rhod}{\Lnorm{2}}^2 \bigr),
	\end{aligned}
\end{equation}
for $ t \in (0,T^*] $, any $ \varepsilon'' \in (0,1) $, where we have used the uniform estimate in \eqref{unest:total}. Thus, after substituting \eqref{ctest:001} and choosing the smallness of 
parameters in the order of $ \varepsilon'', \varepsilon', T^* $, one can conclude that
\begin{equation}\label{ctest:conclusion}
	\begin{aligned}
	& \sup_{0\leq s \leq t} \norm{\delta \vec u(s), \delta \mathfrak T(s), \lbrace \delta \mathfrak q_{\mrm j}(s) \rbrace_{\mrm j \in \moistidx}}{\Lnorm{2}}^2 \\
	& ~~~~ + \int_0^t \norm{\delta \vec u(s),  \delta \mathfrak T(s), \lbrace \delta  \mathfrak q_\mrm j(s) \rbrace_{\mrm j \in \moistidx}}{\Hnorm{1}}^2 \,ds \\
	& \leq \dfrac{1}{2}\sup_{0\leq s\leq t}\norm{\delta \vec u^o(s),  \delta \mathfrak T^o(s), \lbrace  \delta \mathfrak q_{\mrm j}^o(s) \rbrace_{\mrm j \in \moistidx}}{\Lnorm{2}}^2 \\
	& ~~~~ + \dfrac{1}{2} \int_0^t \norm{\delta \vec u^o(s),  \delta \mathfrak T^o(s), \lbrace  \delta \mathfrak q_{\mrm j}^o(s) \rbrace_{\mrm j \in \moistidx}}{\Hnorm{1}}^2 \,ds,
	\end{aligned}
\end{equation}
for $ t \in (0,T^{*}] $. 
Therefore $ \mathcal M $, defined in \eqref{def:mapping}, is a {contraction} mapping in $ \mathfrak M_{T^{*}} $ for such a choice of $ T^* $. 

Consequently, the Banach fixed point theorem implies that there exists a unique fixed point of $ \mathcal M $ in $ \mathfrak M_{T^*} $. 
Let the fixed point of $ \mathcal M $ be $ (\vec u, \mathfrak T, \lbrace \mathfrak q_\mrm j \rbrace_{\mrm j \in \moistidx}) $, and let $ \rhod $ be the unique solution to \eqref{eq:dry-air-mass-app}. Then
$
(\rhod, \vec u, \mathfrak T, \lbrace \mathfrak q_\mrm j \rbrace_{\mrm j \in \moistidx})
$
is the unique solution to \eqref{eq:dry-air-mass-app} -- \eqref{eq:momentum-app} and \eqref{eq:temperature-app-1} -- \eqref{eq:rain-ratio-app-1} with \eqref{bc:momentum} and \eqref{bc:app-hom}. Then by writing, according to \eqref{def:homo-variable}, 
\begin{equation}\label{re-def:homo-variable}
\T = B_\T^{-1}(\mathfrak T + \psi_\T), ~~ q_\mrm j = B_\mrm j^{-1}(\mathfrak q_\mrm j + \psi_\mrm j), ~~~~ \mrm j \in \moistidx,
\end{equation}
$
(\rhod, \vec u, \T, \lbrace  q_\mrm j \rbrace_{\mrm j \in \moistidx})
$
is the unique solution to \eqref{eq:dry-air-mass-app} -- \eqref{eq:rain-ratio-app} with \eqref{bc:momentum} -- \eqref{bc:moisture}.

\section{Well-posedness and non-negativity of the mixing ratios}\label{sec:well-posedness}

In this section, we show first, the unique solution constructed in the last section is stable. Therefore system \eqref{eq:dry-air-mass-app} -- \eqref{eq:rain-ratio-app} with \eqref{bc:momentum} -- \eqref{bc:moisture} is well-posed. The second part of this section is to show the non-negativity of the solution, and thus it is the solution to our original system \eqref{eq:dry-air-mass} -- \eqref{eq:rain-ratio}.

\subsection{Local-in-time well-posedness}

Consider two solutions $ \lbrace (\rho_{\mrm d, i}, \vec u_{i}, \T_{i}, \lbrace q_{\mrm j, i}\rbrace_{\mrm j \in \moistidx} ) \rbrace_{i=1, 2} $ to \eqref{eq:dry-air-mass-app} -- \eqref{eq:rain-ratio-app}, with initial data $ \lbrace(\rho_{\mrm d, in, i}, \vec u_{in, i}, \mathcal T_{in, i}, \lbrace q_{\mrm j, in, i} \rbrace_{\mrm j \in \moistidx}) \rbrace_{i = 1, 2} $, respectively. 

According to \eqref{def:homo-variable} (or \eqref{re-def:homo-variable}), it is equivalent to consider two sets of solutions $ \lbrace (\rho_{\mrm d, i}, \vec u_{i}, \mathfrak T_{i}, \lbrace \mathfrak q_{\mrm j, i}\rbrace_{\mrm j \in \moistidx}) \rbrace_{i=1, 2} $ to \eqref{eq:dry-air-mass-app} -- \eqref{eq:momentum-app} and \eqref{eq:temperature-app-1} -- \eqref{eq:rain-ratio-app-1}, with corresponding initial data.

Let $ \delta(\cdot) $ be the action of taking difference as in section \ref{sec:contracting}. Then estimates \eqref{ctest:density} -- \eqref{ctest:conclusion} still hold for $ \lbrace (\vec u^o_i, \mathfrak T^o_i, \lbrace \mathfrak q_{\mrm j, i}^o \rbrace_{\mrm j \in \moistidx}) \rbrace_{i=1,2} $ replaced by $ \lbrace (\vec u_i, \mathfrak T_i, \lbrace \mathfrak q_{\mrm j, i} \rbrace_{\mrm j \in \moistidx}) \rbrace_{i=1,2} $, with the addition of initial data on the right-hand side of \eqref{ctest:001}, \eqref{ctest:002}, and \eqref{ctest:conclusion}. That is
\begin{gather}
	\begin{aligned}
	& \sup_{0\leq s \leq t} \norm{\delta \rhod(s)}{\Lnorm{2}}^2 \leq c_1 \norm{\delta\rho_{in}}{\Lnorm{2}}^2 + \int_0^t \norm{\delta \vec u(s)}{\Hnorm{1}}^2 \,ds,
	\end{aligned} \label{stability:001}\\
	\begin{aligned}
		& \sup_{0\leq s \leq t} \norm{\delta \vec u(s), \delta \mathfrak T(s), \lbrace \delta \mathfrak q_{\mrm j}(s) \rbrace_{\mrm j \in \moistidx}}{\Lnorm{2}}^2 \\
	& ~~~~ + \int_0^t \norm{\delta \vec u(s),  \delta \mathfrak T(s), \lbrace \delta  \mathfrak q_\mrm j(s) \rbrace_{\mrm j \in \moistidx}}{\Hnorm{1}}^2 \,ds \\
	& \leq c_2 \norm{\delta \vec u_{in}, \delta \mathfrak T_{in}, \lbrace \delta \mathfrak q_{\mrm j,in}\rbrace_{\mrm j \in \moistidx}}{\Lnorm{2}}^2.
	\end{aligned}
	\label{stability:002}
\end{gather}
Therefore, one concludes, from \eqref{stability:001} and \eqref{stability:002}, the uniqueness and the continuous dependency on initial data for solutions to \eqref{eq:dry-air-mass-app} -- \eqref{eq:momentum-app} and \eqref{eq:temperature-app-1} -- \eqref{eq:rain-ratio-app-1}. Therefore this system is well-posed, so as \eqref{eq:dry-air-mass-app} -- \eqref{eq:rain-ratio-app} with \eqref{bc:momentum} -- \eqref{bc:moisture} and corresponding initial data.

\subsection{Non-negativity}\label{sec:non-negativity}
In this section, we {show} that the solutions $ \T, \lbrace q_\mrm{j} \rbrace_{\mrm j \in \moistidx} $ to \eqref{eq:temperature-app} -- \eqref{eq:rain-ratio-app} with non-negative initial data {remain} non-negative.

We remind the reader (see \eqref{def:sign-op}) that
$$
f^- := \dfrac{|f| - f}{2}, ~ f^+ := \dfrac{|f| + f}{2}.
$$
Also, $ \mathfrak c_o $, defined in \eqref{unest:total}, is a uniform constant depending only on $ \mathfrak c_\mrm{in} $.

\subsubsection*{Non-negativity of $ \T $}


Taking the $ L^2 $-inner product of \eqref{eq:temperature-app} with $ - \T^- $ yields, after applying integration by parts, that
\begin{align*}
& \dfrac{1}{2} \dfrac{d}{dt}\norm{(Q_{\mrm{th}}^+)^{1/2} \T^-}{\Lnorm{2}}^2 + \kappa \norm{\nabla \T^-}{\Lnorm{2}}^2 =-  \kappa \bigl(\int_{2\mathbb{T}^2}\partial_z \T \T^- \,d S\bigr)(z)\big|_{z=0}^1 \\
& ~~~~  + \dfrac{1}{2} \int \dt Q_\mrm{th}^+ |\T^-|^2\idx - \int Q_{\mrm{th}}^+ (\vec{u} \cdot \nabla) \T^- \T^- \idx + \int c_\mrm{l} \qr \Vr \dz \T^- \T^- \idx  \\
& ~~~~ + \int Q_\mrm{cp} \dv \vec u |\T^-|^2 \idx - Q_1 \int (\Sev^+ - \Scd^+) |\T^-|^2 \idx \\
& ~~~~ + Q_2 \int (\Sev^+ - \Scd^+) \T^- \idx =: \sum_{i=1}^{7}R_{\T,i}.
\end{align*}
Substituting  \eqref{bc:temperature} into $ R_{\T, 1} $ yields
\begin{gather*}
R_{\T,1} = - \kappa \bigl(\int_{2\mathbb{T}^2} \alpha_{\T}^b(\T^b - \T) \T^- \,dS \bigr)(z)\big|_{z=0}^1 \\
= - \kappa  \bigl( \int_{2\mathbb T^2} \alpha_{\T}^b (\T^b \T^- + |\T^-|^2 ) \,dS \bigr)(z)\big|_{z=0}^1  \leq 0,
\end{gather*}
where condition \eqref{bc:sign} is employed. 
After applying H\"older's inequality, Sobolev embedding inequality, and Young's inequality, one can derive, for $ \forall \delta \in (0,1) $, that
\begin{align*}
& \begin{aligned}
& R_{\T, 2} + R_{\T, 5} \lesssim \norm{\dt Q_\mrm{th}^+, Q_\mrm{cp} \dv \vec u}{\Lnorm{2}}\norm{\T^-}{\Lnorm{3}} \norm{\T^-}{\Lnorm{6}} \\
& ~~~~ \leq  \mathcal N (\mathfrak c_o) \norm{\T^-}{\Lnorm{2}}^{1/2} \norm{\T^-}{\Hnorm{1}}^{3/2} \leq \delta \norm{\nabla T^-}{\Lnorm{2}}^2 + C_\delta \mathcal N(\mathfrak c_o) \norm{\T^-}{\Lnorm{2}}^2, 
\end{aligned}\\
& \begin{aligned}
& R_{\T,3} + R_{\T,4} \lesssim \norm{Q_\mrm{th}^+ \vec u, \qr \Vr }{\Lnorm{\infty}} \norm{\nabla\T^-}{\Lnorm{2}} \norm{\T^-}{\Lnorm{2}} \\
& ~~~~ \leq \delta \norm{\nabla T^-}{\Lnorm{2}}^2 + C_\delta \mathcal N(\mathfrak c_o) \norm{\T^-}{\Lnorm{2}}^2,
\end{aligned}\\
& \begin{aligned}
& R_{\T,6} \leq \mathcal N(\mathfrak c_o) \norm{\T^-}{\Lnorm{2}}^2.
\end{aligned}
\end{align*}
Moreover, since $ \T^+ \cdot \T^- = 0, ~ q_\mrm{vs} \cdot \T^- = 0 $, one has
\begin{align*}
& R_{\T,7} = - Q_2 \int ( c_\mrm{cd} \qv^+ \qc^+ + c_\mrm{cn}(\qv - q_\mrm{vs})^+ q_{\mrm{cn}} ) \T^- \leq 0. 
\end{align*}
Therefore, one can conclude that, after choosing $ \delta $ small enough, 
\begin{equation*}
\dfrac{d}{dt} \norm{(Q_{\mrm{th}}^+)^{1/2} \T^-}{\Lnorm{2}}^2 \leq \mathcal N(\mathfrak c_o) \norm{(Q_{\mrm{th}}^+)^{1/2} \T^-}{\Lnorm{2}}^2,
\end{equation*}
which implies, after applying Gr\"onwall's inequality,
\begin{equation*}
\norm{(Q_{\mrm{th}}^+)^{1/2} \T^-(t)}{\Lnorm{2}}^2 \equiv 0,
\end{equation*}
provided $ \T_\mrm{in} \geq 0 $. 
Thus, we have proved $ \T \geq 0 $. 

\subsubsection*{Non-negativity of $ q_\mrm j $, $ \mrm j \in \moistidx $}

Similarly, taking the $ L^2 $-inner product of \eqref{eq:vapor-ratio-app}, \eqref{eq:cloud-ratio-app}, \eqref{eq:rain-ratio-app}, with $ - \qv^-, - \qc^-, - \qr^- $, respectively, yields, after applying integration by parts, that

\begin{align*}
& \dfrac{1}{2} \dfrac{d}{dt} \norm{ \lbrace q_\mrm{j}^- \rbrace_{\mrm j \in \moistidx} }{\Lnorm{2}}^2 + \norm{\lbrace \nabla q_\mrm{j}^- \rbrace_{\mrm j \in \moistidx}}{\Lnorm{2}}^2 \\
& ~~~~ = - \sum_{\mrm j \in \moistidx}\bigl( \int_{2\mathbb{T}^2} \partial_z q_\mrm j q_\mrm j^- \,dS \bigr)(z)\big|_{z=0}^1 -  \sum_{\mrm j \in \moistidx} \int (\vec u \cdot \nabla) q_\mrm{j}^- q_\mrm j^- \idx\\
& ~~~~ + \int (\dz \Vr + \Vr \dz \log \rhod) |\qr^-|^2 \idx + \int \Vr \dz \qr^- \qr^- \idx \\
& ~~~~ - \int (\Sev^+ - \Scd^+)\qv^- \idx - \int [ \Scd^+ - ( \Sac^+ + \Scr^+) ] \qc^- \idx \\
& ~~~~ - \int [ (\Sac^+ + \Scr^+) - \Sev^+] \qr^- \idx = \sum_{i=8}^{14} R_{q,i}.
\end{align*}
Similarly to $ R_{\T,1} $, one has $ R_{q,8} \leq 0 $. Similarly to $ R_{\T,2} $--$ R_{\T, 4} $, for $ \forall \delta \in (0,1) $, one has
$$
R_{q,9} + R_{q,10} + R_{q,11} \leq \delta \norm{\lbrace \nabla q_\mrm j^- \rbrace_{\mrm j \in \moistidx}}{\Lnorm{2}}^2 + C_\delta \mathcal N(\mathfrak c_o) \norm{\lbrace q_\mrm j^- \rbrace_{\mrm j \in \moistidx}}{\Lnorm{2}}^2.
$$ 
Next, direct calculation yields
\begin{align*}
& \begin{aligned}
& R_{q,12} = - \int \underbrace{\Sev^+}_{\geq 0} \qv^- \idx + \int [ c_\mrm{cd} (\qv^+ - q_{\mrm vs})\qc^+ +  c_\mrm{cn}(\qv - q_\mrm{vs})^+ q_\mrm{cn} ] \qv^- \idx \\
& ~~~~ \leq c_\mrm{cn} \int (\underbrace{\qv - q_\mrm{vs}}_{\leq 0 ~ \text{in} ~ \lbrace \qv < 0 \rbrace})^+ q_\mrm{cn} \qv^- \idx = 0,
\end{aligned}\\
& \begin{aligned}
& R_{q,13} = - \int c_\mrm{cd} (\qv^+ - q_{\mrm vs})\qc^+ \qc^- \idx - \int c_{\mrm{cn}} (\qv - q_\mrm{vs})^+ q_\mrm{cn} \qc^- \idx \\
& ~~~~ + \int c_\mrm{ac}(\underbrace{\qc - q_\mrm{ac}}_{\leq 0 ~ \text{in} ~ \lbrace \qc < 0 \rbrace})^+ \qc^- \idx + \int c_\mrm{cr} \qc^+ \qr^+ \qc^- \idx \leq 0,
\end{aligned}\\
& \begin{aligned}
& R_{q,14} = - \int \underbrace{(\Sac^+ + \Scr ^+)}_{\geq 0} \qr^- \idx + 0
\leq 0,
\end{aligned}
\end{align*}
where we have used the fact $ f^+ \cdot f^- \equiv 0 $ for $ \forall f $. Therefore, one can conclude that, after choosing $ \delta $ small enough, 
\begin{equation*}
\dfrac{d}{dt} \norm{\lbrace q_\mrm j^- \rbrace_{\mrm j \in \moistidx}}{\Lnorm{2}}^2 \leq \mathcal N(\mathfrak c_o) \norm{\lbrace q_\mrm j^- \rbrace_{\mrm j \in \moistidx}}{\Lnorm{2}}^2,
\end{equation*}
which implies, after applying Gr\"onwall's inequality,
\begin{equation*}
\norm{\lbrace q_\mrm j^- (t) \rbrace_{\mrm j \in \moistidx}}{\Lnorm{2}} \equiv 0,
\end{equation*}
provided $ q_{\mrm j , \mrm{in}} \geq 0 $, $\mrm j \in \moistidx$. 
Thus, we have shown $ q_\mrm j \geq 0, ~ \mrm j \in \moistidx $.

\section{Appendix}\label{appendix}

\subsection{An extension lemma}

\begin{lm}\label{lm:extension}
	Let $ s \geq 0 $ and $ h \in \norm{h}{\bHnorm{s}} $, where $ \Gamma := 2\mathbb T^2 \times \lbrace z = 0,1\rbrace = \partial \Omega $ is the boundary of the domain $ \Omega $.
	Then there exists an extension operator $ \Psi $ such that $ \dz \Psi(h)\big|_{z=0,1} = h$ and
	\begin{equation}\label{est:extension}
	\norm{\Psi({h})}{\Hnorm{s+3/2}} \leq C_{\chi_0,s} \norm{h}{\bHnorm{s}},
	\end{equation}
where the positive constant $C_{\chi_0,s}$ depends on $s$ and the cut-off function $\chi_0$ defined, below, in \eqref{def:cut-off}.
\end{lm}

\begin{proof}
After writing $ h $ in terms of the Fourier expansion in horizontal variables, 
$$
h(\vec{x}_h,z,t) = \sum_{\vec{k} \in \mathbb{Z}^2 } \beta_\vec{k}^z(t) e^{i \pi \vec{k}\cdot \vec{x}_h}, ~ z \in \lbrace 0, 1 \rbrace. 
$$
Next, define $ \Phi_0(h) $ and $ \Phi_1(h) $ by
\begin{equation*}
\begin{aligned}	
\Phi_0(h):= & \beta^0_{(0,0)}(t)z - \sum_{\vec{k} \in \mathbb{Z}^2, |\vec{k}| \neq 0 } \dfrac{\beta_\vec{k}^0(t)}{|\vec{k}|} e^{i \pi \vec{k}\cdot \vec{x}_h} e^{-|\vec{k}| z }, \\
\Phi_1(h):= & \beta^1_{(0,0)}(t)z + \sum_{\vec{k} \in \mathbb{Z}^2, |\vec{k}| \neq 0 } \dfrac{\beta_\vec{k}^1(t)}{|\vec{k}|} e^{i \pi \vec{k}\cdot \vec{x}_h} e^{-|\vec{k}| (1-z) }.
\end{aligned}
\end{equation*}
Also, define a cut-off function $ \chi_0 = \chi_0(z) $ by
\begin{equation}\label{def:cut-off}
\chi_0(z) = \begin{cases}
1, & z \in (0, 1/4),\\
\text{smooth, {monotonically} decreasing}, & z \in (1/4, 3/4),\\
0, & z\in (3/4, 1).
\end{cases}
\end{equation}
Then $ \Psi(h) $ is defined by
\begin{equation}\label{def:extension}
\Psi(h) := \chi_0 \Phi_0(h) + (1-\chi_0) \Phi_1(h). 
\end{equation}
Then $ \Psi({h}) $ is smooth, and satisfies \eqref{est:extension}.
\end{proof}

\subsection{Estimates of source terms in the linear system}

One can write, 
\begin{align*}
&\begin{aligned}
& p^o = \rhod ( \cb + \cb \mathfrak q_\mrm{v}^o) ( 1 + \cb \mathfrak T^o ),
\end{aligned}\\ 
& \begin{aligned}
& \mathcal I_{\vec u, 1} = \rhod ( Q_\mrm{m}^{+,o} + \cb \mathfrak q_{\mrm{r}}^o + \cb ) (\vec u^o + \cb \Vr  \vec{e}_3 ) \cdot \nabla \vec u^o,
\end{aligned}\\ 
& \begin{aligned}
& \mathcal I_{\vec u, 2}= - \rhod Q_\mrm{m}^{+,o} g \vec{e}_3.
\end{aligned}
\end{align*}

{\noindent \bf {Source} terms for the parabolic estimates of $ \vec u $:}

\begin{align}
& \begin{aligned}
& \norm{\rhod^{-1/2} \nabla p^o, \dt p^o}{\Lnorm{2}} 
\lesssim \norm{\nabla \rhod^{1/2}, \dt \rhod}{\Lnorm{2}}  ( 1 + \norm{\mathfrak q_\mrm{v}^o}{\Lnorm{\infty}} ) \\ & ~~~~ ~~~~ \times ( \norm{\mathfrak T^o}{\Lnorm{\infty}} +1 )
+ \norm{\rhod^{1/2}, \rhod}{\Lnorm{\infty}} ( 1 + \norm{ \mathfrak q_\mrm{v}^o, \mathfrak T^o}{\Lnorm{\infty}}) \\
& ~~~~ ~~~~ \times \norm{\nabla \mathfrak q_\mrm{v}^o, \nabla \mathfrak T^o, \dt \mathfrak q_\mrm{v}^o, \dt \mathfrak T^o}{\Lnorm{2}}, \end{aligned} \label{ap-001} \\ & \begin{aligned}
& \norm{\rhod^{-1/2} \mathcal I_{\vec u, 1}, \rhod^{-1/2} \mathcal I_{\vec u, 2}}{\Lnorm{2}} \lesssim (\norm{\rhod^{1/2} Q_\mrm{m}^{+,o},\rhod^{1/2}\mathfrak q_{\mrm r}^o}{\Lnorm{\infty}} + 1) \\& ~~~~ ~~~~ \times ( \norm{\vec{u}^o}{\Lnorm{3}} \norm{\nabla \vec{u}^o}{\Lnorm{6}} + \norm{\nabla\vec u^o}{\Lnorm{2}} + 1)\\
& ~~~~  \lesssim (\norm{\lbrace \rhod^{1/2} \mathfrak q_{\mrm j}^o\rbrace_{\mrm j\in \moistidx} }{\Lnorm{\infty}} + 1 ) 
\times ( \norm{\vec{u}^o}{\Hnorm{1}} \norm{\nabla \vec{u}^o}{\Hnorm{1}} \\
& ~~~~ ~~~~ + \norm{\vec u^o}{\Hnorm{1}} + 1), \end{aligned} \label{ap-002}  \\ & \begin{aligned}
& \norm{\rhod^{-1}\dt(\rhod Q_\mrm{m}^{+,o} )}{\Lnorm{5}} \lesssim \norm{Q_\mrm{m}^{+,o}}{\Lnorm{\infty}} \norm{\dt \log \rhod}{\Lnorm{5}}  \\
&  ~~~~ + \norm{\dt Q_\mrm{m}^{+,o}}{\Lnorm{5}} \lesssim (1 + \norm{\rhod ,  \lbrace{\mathfrak q_{\mrm{j}}^o}\rbrace_{\mrm{j} \in \moistidx} }{\Lnorm{\infty}} ) \\
& ~~ \times \norm{\dt \log\rhod,  \lbrace{\dt \mathfrak q_{\mrm{j}}^o} \rbrace_{\mrm{j} \in \moistidx}} {\Lnorm{5}} , \end{aligned} \label{ap-003}  \\ & \begin{aligned}
& \norm{\rhod^{-1/2} \dt \mathcal I_{\vec{u},1}, \rhod^{-1/2} \dt \mathcal I_{\vec u, 2} }{\Lnorm{3/2}} \lesssim (\norm{ \rhod^{1/2}, \lbrace \mathfrak q_\mrm{j}^o \rbrace_{\mrm{j}\in \moistidx}}{\Lnorm{\infty}} + 1)\\
& ~~~~ ~~~~ \times \norm{\dt \rhod^{1/2}, \lbrace \dt \mathfrak q_\mrm{j}^o \rbrace_{\mrm{j}\in \moistidx}}{\Lnorm{2}} ( (\norm{\vec{u}^o}{\Lnorm{\infty}} + 1) \norm{\nabla \vec u^o}{\Lnorm{6}} + 1 )\\
& ~~~~ + \norm{\lbrace \mathfrak q_\mrm{j}^o \rbrace_{\mrm{j}\in \moistidx}}{\Lnorm{\infty}}  \bigl( \norm{\rhod^{1/2}\vec u^o, \rhod^{1/2}}{\Lnorm{6}} \norm{ \nabla\dt \vec{u}^o}{\Lnorm{2}}  \\
& ~~~~ ~~~~ + (\norm{\rhod^{1/2}\dt \vec u^o}{\Lnorm{6}} + 1)  \norm{\nabla \vec u^o }{\Lnorm{2}}  \bigr).
\end{aligned} \label{ap-004}  \end{align}

{\noindent \bf  Source terms for  the elliptic estimates of $\vec u $:}

\begin{align}
& \begin{aligned}
& \norm{\nabla p^o}{\Lnorm{2}} \lesssim \norm{\nabla \rhod}{\Lnorm{6}}(\norm{\mathfrak q_\mrm{v}^o}{\Lnorm{6}} + 1)(\norm{\mathfrak T^o}{\Lnorm{6}} + 1) \\
& ~~~ ~ + \norm{\rhod}{\Lnorm{\infty}}(\norm{\mathfrak q_\mrm{v}^o, \mathfrak T^o}{\Lnorm{6}} + 1) \norm{\nabla \mathfrak q_\mrm{v}^o,  \nabla \mathfrak T^o}{\Lnorm{3}},
\end{aligned}\label{ap-101}  \\ & \begin{aligned}
& \norm{\mathcal I_{\vec u, 1},\mathcal I_{\vec u, 2}}{\Lnorm{2}} \lesssim \norm{\rhod}{\Lnorm{\infty}}\bigl( \norm{\moistvar}{\Lnorm{\infty}}+1\bigr) \\
& ~~~~ \times \bigl((\norm{\vec u^o}{\Lnorm{6}}+1) \norm{\nabla \vec u^o}{\Lnorm{3}} + 1\bigr),
\end{aligned}\label{ap-102} \\ & \begin{aligned}
& \norm{\nabla^2 p^o}{\Lnorm{2}} \lesssim \norm{\rhod, \mathfrak q_\mrm{v}^o, \mathfrak T^o }{\Lnorm{\infty}} \norm{\nabla \mathfrak T^o, \nabla \rhod }{\Lnorm{6}}\norm{\nabla \mathfrak q_\mrm{v}^o,  \nabla \mathfrak T^o }{\Lnorm{3}} \\
& ~~~~ + \norm{\rhod, \mathfrak q_\mrm{v}^o, \mathfrak T^o}{\Lnorm{\infty}}^2 \norm{\nabla^2 \rhod, \nabla^2 \mathfrak q_\mrm{v}^o, \nabla^2 \mathfrak T^o}{\Lnorm{2}},
\end{aligned} \label{ap-103}\\ & \begin{aligned}
& \norm{\nabla \mathcal I_{\vec{u},1},\nabla \mathcal I_{\vec{u},2}}{\Lnorm{2}} \lesssim (\norm{\vec u^o}{\Hnorm{2}}^2 + 1) \norm{\rhod, \lbrace \mathfrak q_\mrm{j}^o \rbrace_{\mrm{j}\in \moistidx}}{\Lnorm{\infty}}  \\
& ~~~~ \times ( \norm{\rhod, \lbrace \mathfrak q_\mrm{j}^o \rbrace_{\mrm{j}\in \moistidx}}{\Lnorm{\infty}} + \norm{\nabla \rhod, \lbrace \nabla \mathfrak q_\mrm{j}^o \rbrace_{\mrm{j}\in \moistidx}}{\Lnorm{3}} ),
\end{aligned} \label{ap-104}
\end{align}

Similarly, one can easily check that,
\begin{align*}
& \begin{aligned}
& \mathcal I_{\T, 1} = ( Q_\mrm{th}^{+,o} + \cb \mathfrak q_\mrm{r}^o + \cb ) (\vec u^o + \cb \Vr \vec{e}_3 ) \cdot \nabla \mathfrak T^o \\
& ~~~~ ~~~~  + Q_\mrm{cp}^o (\mathfrak T^o + \cb ) \dv \vec u^o ,
\end{aligned}	\\
& \begin{aligned}
& \mathcal I_{\T, 2}' = Q_\mrm{th}^{+,o} ( \cb \mathfrak T^o w^o + \cb w^o + \vec{u}^o \cdot\nabla \cb + \cb  ) + \cb \mathfrak q_\mrm{r}^o \mathfrak T^o \\
& ~~~~ ~~~~ + \cb \mathfrak q_\mrm{r}^o + \cb \mathfrak T^o + \cb ,
\end{aligned} \\
& \begin{aligned}
& \mathcal I_{\mrm{j},1} = (\cb \vec{u}^o + \cb \vec{e}_3 ) \cdot \nabla \mathfrak q_\mrm{j}^o, 
\end{aligned} \\
& \begin{aligned}
& \mathcal I_{\mrm{j},2}' = ( \mathfrak q_\mrm{j}^o + \cb )\vec{u}^o \cdot \nabla \cb + \cb \mathfrak q_\mrm{j}^o + \cb.
\end{aligned} 
\end{align*}

{\noindent \bf Source terms for the estimates of thermo-moisture variables, part 1:}
\begin{align}
& \begin{aligned}
& \norm{\mathcal I_{\T, 1} }{\Lnorm{2}} \lesssim ( 1+ \norm{\lbrace \mathfrak q_\mrm{j}^o \rbrace_{\mrm{j} \in \moistidx}}{\Lnorm{\infty}}) \\
& ~~~~ ~~~~ \times (1 + \norm{\vec u^o, \mathfrak T^o}{\Lnorm{6}} \norm{\nabla\mathfrak T^o, \nabla \vec u^o }{\Lnorm{3}}  ),
\end{aligned}\label{ap-201} \\ & \begin{aligned}
& \norm{\lbrace \mathcal I_{\mrm{j},1} \rbrace_{\mrm{j}\in \moistidx }  }{\Lnorm{2}} \lesssim (1 + \norm{\vec u^o}{\Lnorm{6}}) \norm{\lbrace \nabla \mathfrak q_\mrm{j}^o \rbrace_{\mrm{j} \in \moistidx} }{\Lnorm{3}},
\end{aligned}\label{ap-202}\\& \begin{aligned}
& \norm{\dt \mathcal I_{\T, 1}, \lbrace \dt  \mathcal I_{\mrm{j},1} \rbrace_{\mrm{j}\in \moistidx }  }{\Lnorm{3/2}} \lesssim \norm{ \lbrace \dt \mathfrak q_\mrm{j}^o  \rbrace_{\mrm{j}\in \moistidx} }{\Lnorm{2}} \\
& ~~~~ ~~~~ \times (1 + \norm{\vec{u}^o, \mathfrak T^o}{\Lnorm{\infty}}) \norm{\nabla \mathfrak T^o, \nabla \vec u^o}{\Lnorm{6}} \\
& ~~~~ + \bigl( 1+ \norm{\lbrace \mathfrak q_\mrm{j}^o\rbrace_{\mrm j \in \moistidx}}{\Lnorm{\infty}}\bigr) \bigl(  ( \norm{\dt \mathfrak T^o }{\Lnorm{2}} + 1) \norm{\nabla \vec u^o}{\Lnorm{6}} \\
& ~~~~ ~~~~ + \norm{\nabla \mathfrak T^o, \lbrace \nabla \mathfrak q_\mrm{j}^o \rbrace_{\mrm j \in \moistidx}}{\Lnorm{2}} ( \norm{\dt \vec u^o }{\Lnorm{6}} + 1) \bigr) \\
& ~~~~ + ( 1 + \norm{\lbrace \mathfrak q_\mrm{j}^o \rbrace_{\mrm j \in \moistidx}}{\Lnorm{\infty}})(1 + \norm{\vec u^o, \mathfrak T^o }{\Lnorm{6}})\\
& ~~~~ ~~~~ \times \norm{\nabla\dt \mathfrak T^o, \nabla \dt \vec u^o, \lbrace \nabla\dt \mathfrak q_\mrm{j}^o \rbrace_{\mrm j \in \moistidx}}{\Lnorm{2}}, \\
\end{aligned}\label{ap-203} \\& \begin{aligned}
& \norm{\nabla \mathcal I_{\T, 1} , \lbrace \nabla  \mathcal I_{\mrm{j},1} \rbrace_{\mrm{j}\in \moistidx }  }{\Lnorm{2}} \lesssim ( \norm{\lbrace \nabla \mathfrak q_\mrm{j}^o \rbrace_{\mrm j \in \moistidx}}{\Lnorm{3}} + 1) \\
& ~~~~ ~~~~ \times ( 1+ \norm{\vec u^o, \mathfrak T^o}{\Lnorm{\infty}} ) \norm{\nabla \mathfrak T^o, \nabla \vec u^o }{\Lnorm{6}}\\
& ~~~~ + ( 1 + \norm{\lbrace \mathfrak q_\mrm{j}^o\rbrace_{\mrm j \in \moistidx}}{\Lnorm{\infty}}) ( 1 + \norm{\nabla\vec u^o }{\Lnorm{6}})\\
& ~~~~ ~~~~ \times ( 1 + \norm{\nabla \mathfrak T^o, \lbrace \nabla \mathfrak q_\mrm{j}^o \rbrace_{\mrm j \in \moistidx}}{\Lnorm{3}} ) \\
& ~~~~ + ( 1 + \norm{\lbrace \mathfrak q_\mrm{j}^o \rbrace_{\mrm j \in \moistidx}}{\Lnorm{\infty}} ) ( 1 + \norm{\vec u^o, \mathfrak T^o }{\Lnorm{\infty}})\\
& ~~~~ ~~~~ \times \norm{\nabla^2 \mathfrak T^o, \nabla^2 \vec u^o, \lbrace \nabla^2 \mathfrak q_\mrm{j}^o \rbrace_{\mrm j \in \moistidx}}{\Lnorm{2}}.
\end{aligned}\label{ap-204}
\end{align}

{\noindent\bf Source estimates for the estimates of thermo-moisture variables, part 2:}

To estimates, $ \mathcal I_{\mrm{j}, 2} $, $ \mrm{j} \in \lbrace \T, \mrm{v}, \mrm{c}, \mrm r \rbrace $, we need to first deal with $ \mathcal I_{\mrm{j}, 2}' $, $ \mrm{j} \in \lbrace \T, \mrm{v}, \mrm{c}, \mrm r \rbrace $. 

\begin{align}
& \begin{aligned}
& \norm{\lbrace \mathcal I_{\mrm{j}, 2}' \rbrace_{\mrm{j}\in \thmoistidx} }{\Lnorm{2}} \lesssim \norm{\lbrace \mathfrak q_\mrm{j}^o   \rbrace_{\mrm j \in \moistidx}, \mathfrak T^o, \vec u^o }{\Lnorm{6}}^3 + 1,
\end{aligned}\label{ap-301} \\
& \begin{aligned}
& \norm{\lbrace \dt \mathcal I_{\mrm{j}, 2}'\rbrace_{\mrm{j}\in \thmoistidx}}{\Lnorm{3/2}} \lesssim ( \norm{\lbrace \dt  \mathfrak q_\mrm{j}^o \rbrace_{\mrm j \in \moistidx},\dt \mathfrak T^o}{\Lnorm{2}} + 1 ) \\
& ~~~~ ~~~~ \times  ( \norm{\vec u^o}{\Lnorm{6}} + 1) 
( \norm{\moistvar ,\mathfrak T^o}{\Lnorm{\infty}} + 1) \\
& ~~~~ +( \norm{\lbrace \mathfrak q_\mrm{j}^o \rbrace_{\mrm j \in \moistidx}, \mathfrak T^o }{\Lnorm{6}} + 1)^2 ( \norm{\dt \vec u^o}{\Lnorm{3}} + 1),
\end{aligned}\label{ap-302}\\& \begin{aligned}
& \norm{\lbrace \nabla \mathcal I_{\mrm j, 2}'\rbrace_{\mrm{j}\in \thmoistidx}}{\Lnorm{2}} \lesssim ( \norm{\lbrace \nabla \mathfrak q_\mrm{j}^o \rbrace_{\mrm j \in \moistidx}, \nabla  \mathfrak T^o, \nabla \vec{u}^o }{\Lnorm{6}} + 1) \\
& ~~~~ ~~~~ \times ( \norm{\moistvar,\mathfrak T^o,\vec u^o }{\Lnorm{6}} + 1)^2.
\end{aligned}\label{ap-303}
\end{align}
On the other hand, $ \lbrace \mathcal T_{\mrm j,2}- \mathcal T_{\mrm j, 2}'\rbrace_{\mrm j \in \thmoistidx }  $ are Lipschitz functions of $$ \lbrace \mathfrak T^o, \lbrace \mathfrak q_\mrm{j}^o \rbrace_{\mrm j \in \moistidx}, \dz \log \rhod \rbrace. $$
In addition, one can easily check that
\begin{gather*}
| \Sev^{+,o} | \lesssim \cb + \abs{\mathfrak T^o, \mathfrak q_\mrm{r}^o}{2} , ~~~~ 
| \Scd^{+,o} | \lesssim \cb + \abs{\mathfrak q_\mrm{c}^o, \mathfrak q_\mrm{r}^o}{2}, \\
| \Sac^{+,o} | \lesssim  \cb + \abs{\mathfrak q_\mrm{c}^o}{}, ~~~~
| \Scr^{+,o} | \lesssim \cb + \abs{\mathfrak q_\mrm{c}^o, \mathfrak q_\mrm{r}^o}{2}.
\end{gather*}
It is not hard to obtain, with direct calculation, that,
\begin{align}
& \begin{aligned}
& \norm{\lbrace \mathcal T_{\mrm j,2}- \mathcal T_{\mrm j, 2}'\rbrace_{\mrm j \in \thmoistidx }}{\Lnorm{2}} \lesssim 1 + \norm{\mathfrak T^o, \lbrace \mathfrak q_\mrm{j}^o\rbrace_{\mrm j \in \moistidx}, q_\mrm{vs}^o }{\Lnorm{6}}^3\\
& ~~~~ + ( \norm{\mathfrak q_\mrm{r}^o}{\Lnorm{6}} + 1) ( \norm{\dz \log \rhod }{\Lnorm{3}} + 1 ),
\end{aligned}\label{ap-304}\\& \begin{aligned}
& \norm{\lbrace \dt \mathcal T_{\mrm j,2}- \dt \mathcal T_{\mrm j, 2}'\rbrace_{\mrm j \in \thmoistidx }}{\Lnorm{3/2}} \lesssim (1 + \norm{\mathfrak T^o, \lbrace \mathfrak q_\mrm{j}^o  \rbrace_{\mrm j \in \moistidx}}{\Lnorm{\infty}}^2)\\
& ~~~~ ~~~~ \times \norm{\dt \mathfrak T^o, \lbrace \dt \mathfrak q_\mrm{j}^o \rbrace_{\mrm j \in \moistidx}, \dt q_\mrm{vs}^o }{\Lnorm{2}}  + ( \norm{\mathfrak q_\mrm{r}^o }{\Lnorm{\infty}} + 1 ) \\
& ~~~~ ~~~~ \times ( \norm{\dz \dt \log \rhod }{\Lnorm{2}} + 1 ) 
+ ( \norm{\dt \mathfrak q_\mrm{r}^o}{\Lnorm{2}} + 1 )\\
& ~~~~ ~~~~ \times  ( \norm{\dz \log \rhod }{\Lnorm{6}} + 1 ),
\end{aligned}\label{ap-305}\\& \begin{aligned}
& \norm{\lbrace \nabla \mathcal T_{\mrm j,2}- \nabla \mathcal T_{\mrm j, 2}'\rbrace_{\mrm j \in \thmoistidx }}{\Lnorm{2}} \lesssim (1 + \norm{\mathfrak T^o, \lbrace \mathfrak q_\mrm{j}^o  \rbrace_{\mrm j \in \moistidx} }{\Lnorm{\infty}}^2)\\
& ~~~~ ~~~~ \times \norm{\nabla \mathfrak T^o, \lbrace \nabla \mathfrak q_\mrm{j}^o \rbrace_{\mrm j \in \moistidx}, \nabla q_\mrm{vs}^o }{\Lnorm{2}} + (\norm{\mathfrak q_\mrm{r}^o }{\Lnorm{\infty}} + 1) \\
& ~~~~ ~~~~ \times ( \norm{\nabla \dz \log \rhod}{\Lnorm{2}} + 1) +  ( \norm{\nabla \mathfrak q_\mrm{r}^o}{\Lnorm{3}}+ 1) \\
& ~~~~ ~~~~ \times (\norm{\dz \log \rhod }{\Lnorm{6}} + 1).
\end{aligned}\label{ap-306}
\end{align}

\section*{Acknowledgements}
SD thanks the support by the Austrian Science Fund (FWF) (Grant DOI: 10.55776/F65). XL and EST would like to thank the Freie Universit\"{a}t Berlin for the kind hospitality during which part of this work was completed. RK, XL and EST  thank the Deutsche Forschungsgemeinschaft (DFG) for partial funding through Project C06 of CRC 1114 “Scaling Cascades in Complex Systems”, Project Number 235221301, which provided an inspiring environment for their research.
 The work of RK and EST was also supported in part by the DFG Research Unit FOR 5528 on Geophysical Flows.


\begin{thebibliography}{10}

\bibitem{Bousquet2014}
Arthur Bousquet, Michele {Coti Zelati}, and Roger Temam.
\newblock {Phase transition models in atmospheric dynamics}.
\newblock {\em Milan J. Math.}, 82(1):99--128, 2014.

\bibitem{Cho2004}
Yonggeun Cho, Hi~Jun Choe, and Hyunseok Kim.
\newblock {Unique solvability of the initial boundary value problems for
  compressible viscous fluids}.
\newblock {\em J. Math. Pures Appl.}, 83(2):243--275, 2004.

\bibitem{Cho2006a}
Yonggeun Cho and Hyunseok Kim.
\newblock {Existence results for viscous polytropic fluids with vacuum}.
\newblock {\em J. Differ. Equ.}, 228(2):377--411, 2006.

\bibitem{CotiZelati2013}
Michele {Coti Zelati}, Michel Fr{\'{e}}mond, Roger Temam, and Joseph Tribbia.
\newblock {The equations of the atmosphere with humidity and saturation:
  Uniqueness and physical bounds}.
\newblock {\em Phys. D Nonlinear Phenom.}, 264:49--65, 2013.

\bibitem{CotiZelati2012}
Michele {Coti Zelati} and Roger Temam.
\newblock {The atmospheric equation of water vapor with saturation}.
\newblock {\em Boll. dell'Unione Math. Ital.}, 5(9):309--336, 2012.

\bibitem{Emanuel1994}
{Kerry A. Emanuel}.
\newblock {\em {Atmospheric Convection}}.
\newblock Oxford University Press, 1994.

\bibitem{Feireisl2004}
Eduard Feireisl.
\newblock {\em {Dynamics of Viscous Compressible Fluids}}.
\newblock Oxford Lecture Series in Mathematics and its Applications, 26. Oxford
  University Press, 2004.

\bibitem{Frierson2004}
Dargan~M.W. Frierson, Andrew~J. Majda, and Olivier~M. Pauluis.
\newblock {Large scale dynamics of precipitation fronts in the tropical
  atmosphere: A novel relaxation limit}.
\newblock {\em Commun. Math. Sci.}, 2(4):591--626, 2004.

\bibitem{Grabowski1996}
Wojciech~W. Grabowski and Piotr~K. Smolarkiewicz.
\newblock {Two-time-level semi-Lagrangian modeling of precipitating clouds}.
\newblock {\em Mon. Weather Rev.}, 124(3):487--497, 1996.

\bibitem{Hittmeir2018}
Sabine Hittmeir and Rupert Klein.
\newblock {Asymptotics for moist deep convection I: refined scalings and
  self-sustaining updrafts}.
\newblock {\em Theor. Comput. Fluid Dyn.}, 32(2):137--164, 2018.

\bibitem{Hittmeir2017}
Sabine Hittmeir, Rupert Klein, Jinkai Li, and Edriss~S. Titi.
\newblock {Global well-posedness for passively transported nonlinear moisture
  dynamics with phase changes}.
\newblock {\em Nonlinearity}, 30(10):3676--3718, 2017.

\bibitem{Hittmeir2019}
Sabine Hittmeir, Rupert Klein, Jinkai Li, and Edriss~S Titi.
\newblock {Global well-posedness for the primitive equations coupled to
  nonlinear moisture dynamics with phase changes}.
{\newblock {\em Nonlinearity},  33:3206--3236, 2020.}

\bibitem{Hittmeir2023}
Sabine Hittmeir, Rupert Klein, Jinkai Li, and Edriss~S Titi.
\newblock {Global well-posedness for the thermodynamically refined passively transported nonlinear moisture dynamics with phase changes}.
\newblock {\em Journal of Nonlinear Science}, 33:65, 2023.



\bibitem{Kessler1969}
Edwin Kessler.
\newblock {On the distribution and continuity of water substance in atmospheric
  circulations}.
\newblock In {\em Distrib. Contin. Water Subst. Atmos. Circ.}, pages 1--84.
  American Meteorological Society, Boston, MA, 1969.

\bibitem{Klein2006}
Rupert Klein and Andrew~J. Majda.
\newblock {Systematic multiscale models for deep convection on mesoscales}.
\newblock {\em Theor. Comput. Fluid Dyn.}, 20(5-6):525--551, 2006.

\bibitem{Li2017b}
Hailiang Li, Yuexun Wang, and Zhouping Xin.
\newblock {Non-existence of classical solutions with finite energy to the
  Cauchy problem of the compressible Navier--Stokes equations}.
\newblock {\em Arch. Ration. Mech. Anal.}, 232:557--590,  2018.

\bibitem{Li2015}
Jinkai Li and Edriss~S. Titi.
\newblock {A tropical atmosphere model with moisture: global well--posedness
  and relaxation limit}.
\newblock {\em Nonlinearity}, 29(9):2674--2714, 2016.

\bibitem{Lions1998}
Pierre-Louis Lions.
\newblock {\em {Mathematical Topics in Fluid Mechanics. Volume 2. Compressible
  Models}}.
\newblock Oxford Lecture Series in Mathematics and Its Applications , Vol 2, No
  10. Oxford University Press, 1998.


\bibitem{LT2021-CPE}
Xin Liu and Edriss S. Titi.
\newblock{Local Well-Posedness of Strong Solutions to the Three-Dimensional Compressible Primitive Equations}.
\newblock {\em Arch. Rational Mech. Anal.}, 241(2):729-764, 2021. Also at \href{https://arxiv.org/abs/1806.09868}{arXiv:1806.09868}.

\bibitem{Majda2010}
Andrew~J. Majda and Panagiotis~E. Souganidis.
\newblock {Existence and uniqueness of weak solutions for precipitation fronts:
  A novel hyperbolic free boundary problem in several space variables}.
\newblock {\em Commun. Pure Appl. Math.}, 63(10):1351--1361, 2010.

\bibitem{Matsumura1980}
Akitaka Matsumura and Takaaki Nishida.
\newblock {The initial value problem for the equations of motion of viscous and
  heat--conductive gases}.
\newblock {\em J. Math. Kyoto Univ.}, 20(1):67--104, 1980.

\bibitem{Matsumura1983}
Akitaka Matsumura and Takaaki Nishida.
\newblock {Initial boundary value problems for the equations of motion of
  compressible viscous and heat--conductive fluids}.
\newblock {\em Commun. Math. Phys.}, 89:445--464, 1983.

\bibitem{Simon1986}
	{Jacques Simon}.
	\newblock {Compact sets in the space $ L^p(0,T;B) $}.
	\newblock {\em J. Annali di Matematica pura ed applicata}, 146(1):65--96, 1986.

\bibitem{temam1977}
{Roger Temam}.
\newblock {\em Navier-Stokes Equations: Theory and Numerical Analysis}.
\newblock Studies in Mathematics and its Applications. Elsevier Science, 2016.


\bibitem{zpxin1998}
Zhouping Xin.
\newblock {Blowup of smooth solutions to the compressible Navier--Stokes
  equation with compact density}.
\newblock {\em Commun. pure Appl. Math.}, LI:229--240, 1998.

\bibitem{XinYan2013}
Zhouping Xin and Wei Yan.
\newblock {On blowup of classical solutions to the compressible Navier--Stokes
  equations}.
\newblock {\em Commun. Math. Phys.}, 321(2):529--541, 2013.

\bibitem{Zelati2015a}
Michele~Coti Zelati, Aimin Huang, Igor Kukavica, Roger Temam, and Mohammed
  Ziane.
\newblock {The primitive equations of the atmosphere in presence of vapour
  saturation}.
\newblock {\em Nonlinearity}, 28(3):625--668, 2015.


\end{thebibliography}
\bibliographystyle{plain}

\end{document}